\documentclass{amsart}

\input{epsf.tex}

\usepackage{amsfonts,amssymb,verbatim,amsmath,amsthm,latexsym,textcomp,amscd}
\usepackage{graphicx}
\usepackage{color}
\usepackage{url}

\begin{document}

\newtheorem{theorem}{Theorem}[section]
\newtheorem{prop}[theorem]{Proposition}
\newtheorem{proposition}[theorem]{Proposition}
\newtheorem{lemma}[theorem]{Lemma}
\newtheorem{cor}[theorem]{Corollary}
\newtheorem{definition}[theorem]{Definition}
\newtheorem{conj}[theorem]{Conjecture}
\newtheorem{claim}[theorem]{Claim}
\newtheorem{defth}[theorem]{Definition-Theorem}
\newtheorem{example}[theorem]{Example}
\newtheorem{obs}[theorem]{Observation}
\newtheorem{rmk}[theorem]{Remark}
\newtheorem{remark}[theorem]{Remark}
\newtheorem{criterion}[theorem]{Criterion}

\newtheorem{introthm}{Theorem}
\newtheorem{introcor}[introthm]{Corollary}
\renewcommand{\theintrothm}{\Alph{introthm}}
\renewcommand{\theintrocor}{\Alph{introcor}}
\newtheorem{appendthm}{Lemma}[section]
\renewcommand{\theappendthm}{A.\arabic{appendthm}}
\newtheorem{appendremark}[appendthm]{Remark}

\newcommand{\dd}{{\partial}}
\newcommand{\e}{{\epsilon}}
\newcommand{\s}{{\sigma}}

\newcommand\AAA{{\mathcal A}}
\def\Ax{{\mathop {\rm{Ax}}}}
 
\newcommand\BB{{\mathcal B}}
\newcommand\CC{{\mathcal C}}
\newcommand\C{{\mathbb C}}
\newcommand{\Chat}{{\hat {\mathbb C}}}
\newcommand\EE{{\mathcal E}}
\newcommand\G{{\Gamma}}
\newcommand\Gr{{\mathcal G}}
\newcommand\HH{{\mathcal H}}
\newcommand{\HHH}{{\mathbb H}}

\newcommand\Hyp{{\mathbb H}}
\newcommand\GG{{\mathcal G}} 
\newcommand\MM{{\mathcal M}}
\newcommand\FF{{\mathcal F}}
\newcommand\LL{{\mathcal L}}
\newcommand\XX{{\mathcal X}}
\newcommand\YY{{\mathcal Y}}

\newcommand{\N}{{\mathcal N}}
\newcommand\PP{{\mathcal P}}

\newcommand\TT{{\mathcal T}}

\newcommand{\PSL}{{PSL_2 (\mathbb{C})}}
 \newcommand\til{\widetilde}
\newcommand\CH{{\CC\HH}}

\newcommand\EXH{{ \EE (X, \HH )}}
\newcommand\GXH{{ \GG (X, \HH )}}
\newcommand\GYH{{ \GG (Y, \HH )}}
\newcommand\PEX{{\PP\EE  (X, \HH , \GG , \LL )}}
\newcommand\MF{{\MM\FF}}
 
\def\mul{\stackrel{{}_\ast}{\asymp}}
\def\add{\stackrel{{}_+}{\asymp}}
\def\ladd{\stackrel{{}_+}{\prec}}
\def\gadd{\stackrel{{}_+}{\succ}}
\def\lmul{\stackrel{{}_\ast}{\prec}}
\def\gmul{\stackrel{{}_\ast}{\succ}}

\def\co{{\colon \thinspace}}
 \def\Stab {{\mathop {\rm Stab}}}

\title{Limits of  limit sets I}

\author{Mahan Mj}
\address{\begin{flushleft} \rm {\texttt{mahan@rkmvu.ac.in; mahan.mj@gmail.com \\http://people.rkmvu.ac.in/$\sim$mahan/} }\\ School
of Mathematical Sciences, RKM Vivekananda University\\
P.O. Belur Math, Dt. Howrah, WB 711202, India \end{flushleft}}

\author{Caroline Series}
\address{\begin{flushleft} \rm {\texttt{C.M.Series@warwick.ac.uk \\http://www.maths.warwick.ac.uk/$\sim$masbb/} }\\ Mathematics Institute, 
 University of Warwick \\
Coventry CV4 7AL, UK \end{flushleft}}
\thanks{Research of first author partially supported by a DST Research grant}   

\date{\today} 

 \begin{abstract}
 
We show that for a strongly convergent sequence of geometrically finite Kleinian groups with
geometrically finite limit, the Cannon-Thurston maps of limit sets converge uniformly. If however the algebraic and geometric limits differ, as in the well known examples due to Kerckhoff and Thurston, then provided the geometric limit is geometrically finite, the 
  maps on limit sets converge pointwise but not uniformly. 
 \medskip

\end{abstract}

 \maketitle


\section{Introduction}

Hausdorff convergence of limit sets under algebraic and geometric limits has been studied by several authors, see for example \cite{marden-book} p.\;203 and Theorem~\ref{evans} below.
In this paper and its companion~\cite{mjs2}, we study convergence of limit sets as \emph{the  
convergence of a sequence of continuous maps} from a fixed compact set, namely the limit set of a 
 fixed geometrically finite group,  to the sphere.

 Given an isomorphism $\rho \co \Gamma \to G$ between two geometrically finite Kleinian groups,  Floyd~\cite{floyd} showed that there is a continuous  equivariant map from a certain completion of the  Cayley graph   of $\Gamma$ to the limit set $  \Lambda_{G}$.  As long as $\rho$ is weakly type preserving (meaning that 
the image of every parabolic element of $\G$ is also parabolic in $G$), this map factors through  the limit set $  \Lambda_{\G}$, giving a  continuous  equivariant map  $  \Lambda_{\G} \to  \Lambda_{G}$ between the limit sets. This result was extended by the remarkable work of Cannon-Thurston~\cite{CT} who showed that  such a map may still exist even when $G$ is totally degenerate, by giving examples with $\G$ a surface group so that 
$  \Lambda_{\G} = S^1$, and a continuous surjective equivariant map to     $\Lambda_{G} =S^2$.
In fact   maps between limit sets of isomorphic groups have a long history prior to Floyd's paper as so-called \emph{boundary mappings}, going back to Nielsen~\cite{nielsen}, see also~\cite{fenchel-n} Section 25.2 and  for example~\cite{tukia, kuusalo}. We will nevertheless stick with what is now well established terminology and call  an equivariant continuous map between the limit sets $\Lambda_{\Gamma},  \Lambda_{G}$ of two isomorphic Kleinian groups, a Cannon-Thurston  or \emph{CT}-map. 
 
Suppose now that we have a sequence of isomorphisms $\rho_n: \Gamma \to G_n$. This paper is the first of two which studies the pointwise behaviour of the \emph{CT}-maps  as  the sequence $G_n$ converges to a limiting Kleinian group  $G_{\infty}$. In this paper we confine ourselves to the case in which both algebraic and geometric limits of  $(G_n)$  are geometrically finite; in the second~\cite{mjs2} we will use the geometric models given by  Minsky's ending lamination theorem~\cite{minsky-elc1} and results of the first author on existence of Cannon-Thurston maps~\cite{mahan-kl} under very general hypotheses,  to study the situation in which group $G_{\infty}$ is geometrically infinite.  

The results in this paper are the following:

\begin{introthm}\label{thm:strong=unif}  
 Let $\Gamma$ be a fixed geometrically finite   Kleinian group and $\rho_n \co \Gamma  \to G_n$ be a sequence
 of weakly type-preserving
isomorphisms to geometrically finite Kleinian groups $G_n$ which converge strongly to a geometrically finite Kleinian group $G_{\infty} = \rho_{\infty} (\Gamma )$. Then the sequence of CT maps $\hat i_n: \Lambda_{\Gamma} \to  \Lambda_{G_n}$ converges uniformly to $\hat i_{\infty}: \Lambda_{\Gamma} \to  \Lambda_{G_{\infty}}$.
\end{introthm}

\begin{introthm}\label{thm:alg=ptwise}  
 Let $\Gamma$ be a fixed geometrically finite  Kleinian group   and $\rho_n\co \Gamma \to G_n$ be a sequence
 of weakly type-preserving
isomorphisms to geometrically finite Kleinian groups which converge algebraically  to $G_{\infty} = \rho_{\infty} (\Gamma )$. Suppose also that the geometric limit of the groups $G_n$ is geometrically finite. Then the sequence of  {CT}-maps $\hat i_n: \Lambda_{\Gamma} \to  \Lambda_{G_n}$ converge pointwise to $\hat i_{\infty}: \Lambda_{\Gamma} \to  \Lambda_{G_{\infty}}$.
\end{introthm}

We remark that if  the geometric limit of geometrically finite groups is geometrically finite, so is the algebraic limit, see~\cite{marden-book} Theorem 4.6.1 and~\cite{jor-mar}. However geometrically finiteness of the algebraic limit does not in general imply  the same for the geometric limit, as can be seen for example by examining the many geometric limits of once punctured torus groups constructed in~\cite{MinskyOPT}.

If the convergence is algebraic but not strong, then uniform convergence necessarily fails.
This is a consequence of the following result, which in the generality below is due to 
Evans.

\begin{theorem} [\cite{evans1}, \cite{evans2}]   Suppose $\rho_n\co \Gamma \rightarrow G_n$ is a sequence of weakly type-preserving
isomorphisms from a geometrically finite group
$\Gamma$ to Kleinian groups $G_n$ with limit sets $\Lambda_n$, 
such that the sequence converges algebraically to $\rho_\infty \co \Gamma \rightarrow G_\infty$ and geometrically to $H$. Let $\Lambda_{\infty}$ and
 $\Lambda_{H}$ denote the limit sets of $G_\infty$ and  $H$ respectively.
Then $\Lambda_n \rightarrow \Lambda_H$ in the Hausdorff metric. If in addition $\Gamma$ is non-elementary, the sequence converges strongly if and only if
 $\Lambda_n \rightarrow \Lambda_{\infty}$ in the Hausdorff metric.
\label{evans}
\end{theorem}

 For the purposes of this paper, in which the geometric limit is geometrically finite  with non-empty regular set,  then we are in the easier situation of~\cite{jor-mar} Proposition 4.2, see also~\cite{marden-book}  Theorem 4.5.4.

Since uniform convergence implies diagonal convergence of limit points and hence Hausdorff convergence, Theorem~\ref{evans} shows that if $\Lambda_{\infty}$ and
 $\Lambda_H$ differ, then uniform convergence is impossible.

 Theorem~\ref{thm:alg=ptwise}  is illustrated
 by the   Kerckhoff-Thurston examples \cite{ kerckhoff-thurston} of a sequence $G_m$ of quasi-Fuchsian groups converging
geometrically to a geometrically finite group $G$ with a rank $2$-cusp. 
Our proof will make clear  how it is that the \emph{CT}-maps
  converge pointwise, but not uniformly. The lack of uniform convergence in these examples has also been noted in a remark due to Souto, see Section 9 of~\cite{franca}.

As a special case of Theorem~\ref{thm:strong=unif}  one obtains an alternative proof of the following
 well-known application of the $\lambda$-lemma from complex dynamics \cite{MSS}.
Let $A \subset \C$ be a connected open set. A family of quasi-Fuchsian groups $  G_{z}, z \in A$
is said to be \emph{holomorphic} if there are  maps $\rho_z: \G \to G_z, z \in A$ from a fixed (finitely generated) Fuchsian  group $\G$ such that $ z \mapsto \rho_z(g)$ is holomorphic for each $g \in \G$. If the maps $\rho_z$ are all type preserving, then in the formula for the fixed points, the attracting fixed point (or unique parabolic fixed point) is defined  by a fixed branch of  the square root, so that 
 the maps
$i_z: z \mapsto g^+(z)$ are holomorphic. A straightforward application of the $\lambda$-lemma allows one to extend $i_z$ to the entire limit set $\Lambda_{\G}$: 

\begin{theorem} \label{lambdalemma}  Let $ z \to G_{z}$ be a holomorphic family of Kleinian groups such that 
the map $\rho_z: \G \to G_z$ is an isomorphism for $z \in A$ and such that the map $  i_z: \Lambda_G^+ \to \Chat$ is injective. Then the natural embedding 
$  i_z: \Lambda_G^+ \to \Chat$ extends to a continuous equivariant
  homeomorphism $\hat  i_z: \Lambda_G \to \Chat$, such that  $\hat  i_z: \Lambda_G \times A \to \Chat$ is jointly continuous and such that  
 the map 
 $z \mapsto \hat i_z(\xi)$ is holomorphic on $A$ for each $\xi \in \Lambda_{\G}$.
  \end{theorem}
  
Here $\Lambda_G^+$ is the set of attractive fixed points of loxodromic elements in $\G$, see Section~\ref{sec:basicCT}. 
This result applies, for example, in the interior of quasifuchsian deformation space $\mathcal {QF}$. However if  $z \in \partial \mathcal {QF}$ then the map $i_z$ is not in general a bijection in a neighborhood of $z$, so the above result does not give information about what happens as we limit on the boundary $\partial \mathcal {QF}$.

 Miyachi~\cite{miyachi} proved Theorem~\ref{thm:strong=unif} without the restrictions on geometric finiteness, in the case in which $\G$ is a surface group without parabolics and the injectivity radius is bounded along the whole sequence. The only other convergence result  we are aware of is that of Francaviglia~\cite{franca} Theorem 1.8 in which
 $\Gamma$ is a discrete subgroup of $\HHH^k$ which diverges at its critical exponent (and so only applies in our context to $\Gamma$ Fuchsian with no parabolics, see~\cite{marden-book} Theorem 3.14.3). It is shown that, provided that all groups in question are non-elementary, then the corresponding \emph{CT}-maps exist and converge almost everywhere with respect to Patterson-Sullivan measure on  $\Lambda_{\G}$.

The  results  in this paper  are all based on the  geometry of geometrically finite groups. In the geometrically infinite situation of~\cite{mjs2}, besides  the Minsky model of degenerate Kleinian groups, we introduce techniques from the theory of  relative hyperbolic spaces and electric geometry.
We will prove  in~\cite{mjs2} that Theorem~\ref{thm:strong=unif} holds in many cases of geometrically infinite  $G_{\infty}$, but  show that for certain algebraic limits (such as the examples of Brock~\cite{brock-itn}), even pointwise convergence may fail.

In Section~\ref{sec:background} we set up the background
and in Section~\ref{sec:cayley} we   prove some important estimates based on Floyd~\cite{floyd} on the embedding of the Cayley graph of a geometrically finite group  into $\HHH^3$. This is a prelude to Section~\ref{sec:CT}  in which we describe the \emph{CT}-maps carefully and  give a geometrical criterion Theorem~\ref{crit1} for their existence. In Section~\ref{sec:criteria} this  is extended to a criterion Theorem~\ref{unifcrit1} for uniform convergence. In Section~\ref{sec:strong} we prove Theorem~\ref{thm:strong=unif}, the case of strong convergence and easily deduce the application to Theorem~\ref{lambdalemma}. Finally  in Section~\ref{sec:algebraic} we formulate the corresponding criterion for pointwise convergence and prove  Theorem~\ref{thm:alg=ptwise}. Some hyperbolic geometry estimates we need are given in the Appendix.


 \section{Preliminaries}
 
 \label{sec:background}


\subsection{Kleinian groups}
 
 \label{sec:basics}

 A \emph{Kleinian group} $G$ is a discrete subgroup of $PSL_2 (\mathbb{C})$. As such it acts as a properly discontinuous group of isometries of hyperbolic $3$-space $\HHH^3$, whose boundary we identify with  the Riemann sphere $\Chat = \C \cup \infty$.
 All groups considered in this paper will be finitely generated and torsion free, so that  
 $ M = \HHH^3/G$ is a hyperbolic $3$-manifold. An excellent reference for the background we need is~\cite{marden-book}.

 A Kleinian group is \emph{geometrically finite} if it has a fundamental polyhedron in $\HHH^3$ with finitely many faces.
 The limit set $\Lambda_G \subset \Chat$  is the closure of the set of its non-elliptic fixed points.  It can also be defined as  the set of accumulation points of any $G$-orbit.

Let $\mathcal N$ be the hyperbolic convex hull of $\Lambda_G$ in $\HHH^3$. This projects to the convex core of $ M $, that is, the smallest closed convex subset containing all closed geodesics. An alternative characterisation of being geometrically finite is that a $\delta$-neighbourhood of $\mathcal N/G$ in $\HHH^3$ has finite volume for some $\delta>0$. Note that if $G$ is geometrically finite, then $\mathcal N/G$ is compact if and only if $G$ contains no parabolics. Such groups are called \emph{convex cocompact}.

The \emph{thin part} $M_{thin}(\epsilon)$ is by definition that part of $M$ where the injectivity radius is at most $\epsilon$. The Margulis constant $\epsilon_{\mathcal M} >0$ is  such that if $\epsilon \leq \epsilon_{\mathcal M} $, then $M_{thin}(\epsilon)$ is a disjoint union of horocyclic neighbourhoods of cusps and Margulis tubes around short geodesics.
It is well known that 
the $\delta$-neighbourhood of $\mathcal N/G$  has finite volume if and only if the thick part of $(\mathcal N/G )\setminus M_{thin}(\epsilon)$ has finite diameter. (Consider a maximal collection of  disjoint embedded balls of small radius in $(\mathcal N/G ) \setminus M_{thin}(\epsilon)$.)

 We denote by $\HH = \HH_{\epsilon} = \HH_{\epsilon;G}$   the set of  lifts of components of $M_{thin}(\epsilon)$ to $\HHH^3$. Thus $\HH$  is a $G$-invariant collection of horoballs based at the parabolic fixed points of $G$, together with equidistant tubes about geodesic axes whose associated loxodromics are short. If we are dealing with a single geometrically finite group $G$, then by reducing $\epsilon$ we can assume that all elements of $\HH$ are horoballs, see for example~\cite{marden-book} Theorem 3.3.4.
We will denote by $\mathcal V=\mathcal V_{\epsilon} = \mathcal V_{\epsilon;G}$ the $\epsilon$-thick part $ \mathcal N$ relative to $G$, that is, $\mathcal V$ is the closure of  $\mathcal N \setminus (\cup_{H \in \HH} H)$.

Let $\Gamma$ be a fixed Kleinian group.
A representation $\rho: \G \to PSL_2 (\mathbb{C})$ is said to be \emph{type preserving} if it maps loxodromics to loxodromics and parabolics to parabolics. It is   \emph{weakly type preserving} if 
the image of every parabolic element of $\G$ is also parabolic in $G$.

Let $\rho_n\co \G \to PSL_2 (\mathbb{C}), n = 1,2 \ldots  $ be a sequence of group isomorphisms.
 The representations $\rho_n$ are said to converge  to the representation $\rho_{\infty}\co \Gamma \to PSL_2 (\mathbb{C})$  \emph{algebraically}
if  for each $g \in \G$,  $\rho_n(g) \to \rho_{\infty} (g)$ as elements of $PSL_2 (\mathbb{C})$. 
They are said to converge \emph{geometrically}   if  $(G_n = \rho_n(G))$ converges as a sequence of closed subsets of $\PSL$
to  $H \subset PSL_2 (\mathbb{C})$. Then  $H$ is a   Kleinian group called the
\emph{ geometric limit} of $(G_n)$.    The sequence   $(G_n)$ converges  \emph{ strongly }
to  $\rho_{\infty}(G)$ if $\rho_{\infty}(G)= H$ and the convergence is both geometric and algebraic.

\subsection{Cannon-Thurston Maps} \label{sec:basicCT}

Let $\G$ be a Kleinian group and let $\rho\co \G \to \PSL$ with $\rho(\G) = G$.
Let $\Lambda_{\G}, \Lambda_G$ be the corresponding limit sets.
A \emph {Cannon-Thurston} or \emph{CT}-map  is an equivariant continuous map  $\hat i\co \Lambda_{\G} \to  \Lambda_G$, that is, a map such that 
$$ \hat i    (g \cdot \xi)  = \rho(g) \hat i (\xi) \ \ {\rm for \ all} \ \ g \in \G, \xi \in \Lambda_{\G}.$$

For a loxodromic  $A \in \PSL$,  denote by $A^+$ its attracting fixed point.
For simplicity of notation, if $A$ is parabolic, we denote the single fixed point  in the same way.
From the equivariance it easily follows that  a Cannon-Thurston map
preserves fixed points, namely
$$ \hat i    (g^+)  = \rho(g)^+ \ \ {\rm for \ all} \ \ g \in \G.$$
Notice that in this we do not necessarily assume that $\rho$ is  type preserving, but that 
weakly type preserving is clearly a necessary condition for the existence of an \emph{CT}-map.
Denoting by $\Lambda_{\G}^+ $  the subset of attracting fixed  points, we see that the \emph{CT}-map, if it exists, is the continuous extension of 
the above obvious map $\Lambda_{\G}^+ \to \Lambda_{G}^+$ to the whole of $\Lambda_{\G}$.

Here is an alternative view on the construction of \emph{CT}-maps.
The Cayley graph of $\G$ is naturally embedded in $\HHH^3$ by a map $j_{\G}: \Gr \G \to \HHH^3$ which sends the vertex $  g  \in \Gr \G$ to $ g \cdot O$ where $O = O_{G} \in \HHH^3$ is a fixed base point and extends in the obvious way to edges, see Section~\ref{sec:cayley} for details. Suppose that $\rho \co  \G \to G$ is   weakly type
 preserving. Then define $i \co j_{\G}(\Gr \G) \to \HHH^3$ by setting $i ( j_{\G}(g)) =  j_{G}(g)$
 on vertices and again extending in the obvious way.  Clearly $i$ is equivariant in the sense that 
$i(hg\cdot O) = \rho(h)i(g\cdot O)$.
If
the map $i$ extends (with respect to the Euclidean metric in the ball model) to a continuous map $\hat i : \Lambda_{\G} \to \Lambda_G$, it follows easily from the equivariance and continuity that  $\hat i (g^+) = \rho(g)^+  $ for all $g \in \Gamma$, so that $\hat i$ is an \emph{CT}-map as defined above.
(A necessary condition for this extension to work  is that $\rho$ be weakly type preserving.)

The original  interest in the Cannon-Thurston maps applied to the case in which 
$\G$ is a surface group and $G$ a doubly degenerate group which is the cyclic cover
of a $3$-manifolds fibering over the circle with pseudo-Anosov monodromy, \cite{CT}, in which case of course $\Lambda_{\G}$ is a round circle and $\Lambda_{G} = \Chat$.
However we can also put results of Floyd \cite{floyd} in this context. 
Floyd constructed the completion $\bar \G$ of the Cayley graph of $\G$ with a suitable metric and showed:

\begin{theorem} [\cite{floyd}] \label{floyd1} Let $\G$ be a geometrically finite Kleinian group with group completion $\bar \G$. Then there is an equivariant continuous map $\bar \G \setminus \G \to \Lambda_{\G}$ which is $2-1$ on parabolic points and injective elsewhere.
\end{theorem}
 Floyd's  map  gives a continuous extension of the embedding $ i : j_{\G} (\Gr \G) \to \HHH^3$ to a map $\bar \G \to \HHH^3 \cup \dd \HHH^3$. In the special case in which $G$ contains no parabolics, $j_{\G} (\Gr \G)$ is a hyperbolic metric space and $\bar G$  can be identified both with its Gromov boundary and with $\Lambda_{\G}$. 
For precise details of the connection between Floyd's result  and maps of limit sets, see~\cite{tukia1}.

 In Section~\ref{sec:CT} we give a criterion for the existence of a  $CT$-map in the case in which $\G$ is any geometrically finite Kleinian group and $\rho: \G \to G$ is a weakly type preserving isomorphism.  It is not hard to deduce  the above result of Floyd.   In subsequent sections  the criterion is extended  to deal with converging sequences of groups.

\subsection{Notation} We denote the 
 hyperbolic metric on $\HHH^3$ by $d_{\HHH}$; sometimes we explicitly  use the ball model $\mathbb B$ with centre $O$ and denote by $d_{\mathbb E}$   the Euclidean metric  on $\mathbb B \cup \Chat$.
For $P \in \HHH^3$,  write $B_{\HHH}(P; R)$ for the hyperbolic ball centre $P$ and  radius $R$, with a similar definition for Euclidean balls $B_{\mathbb E}$.  
  Let $\beta$ be a path in $\HHH^3$ with endpoints $X,Y $. We denote its hyperbolic  length by $\ell(\beta)$ and    write $[\beta]$  or $[X,Y]$ for the $\HHH^3$-geodesic  from $X$ to $Y $ (so that $\ell(\beta) \geq \ell([\beta])$ with equality if and only if $\beta$ is itself geodesic).

We write $ X \prec Y$  (resp. $ X \succ Y$) to mean there is a constant $c>0$  such that $ X < c Y$  (resp. $ X >c Y$)
and $ X \asymp Y$ to mean $ X \prec Y$ and $Y\prec X$.
We also write $ X \ladd Y$  to indicate an inequality up to an additive constant, thus $ X \ladd Y$ means 
  there is a constant $c>0$  such that  $c>0$  such that $ X < Y+c$; the notation
 $Y\gadd X$ is defined similarly.

\section{Embedding the Cayley graph} \label{sec:cayley}
   Let $G$ be a finitely generated Kleinian group with generating set $G^* = \{e_1, \ldots, e_k \}$. Its Cayley graph $\mathcal GG$ is the graph whose vertices are elements $g \in G$ and which has an edge between $g,g'$ whenever $g^{-1}g' \in G^*$.  The graph metric $d_G$ is defined as the edge length of the shortest path between vertices so that $d_G(1,e_i) =1$ for all $i$, where $1$ is the unit element of $G$.  Let
 $|g|$ (or where needed for  clarity $|g|_{G^*}$)  denote the word length of $g \in   G$ with respect to $G^*$, so that $|g| = d_G(1,g)$. For $X \in \Gr G$, we denote by $B_{G}(X; R) \subset \Gr G$  the $d_G$-ball centre $X$ and  radius $R$.

Choose a base point  $O = O_G \in \HHH^3$ which is not a fixed point of any element of $G$. One may if desired assume the basepoint is the centre of the ball model $\mathbb B$. For simplicity, we do this throughout the paper unless indicated otherwise. 
 Then $\mathcal GG$  is embedded in $\HHH^3$ 
by the map $j_{G}$ which sends $g \in G$ to $j_{G}(g) = g \cdot O$, and which sends the edge joining $g,g'$ to the $ \HHH^3$-geodesic joining $j_{G}(g),j_{G}(g')$. In particular,  $j_G(1) = O$. Note that using the ball model of $\HHH^3$, the limit set $\Lambda_G  $ may be regarded as the completion of $j_{G}(\mathcal GG)$  in the Euclidean metric $d_{\mathbb E} $ on   $\mathbb B \cup \Chat$.  

It will be  important for us to understand the relationship between geodesic segments lying outside large balls in $  {\mathcal G}G$ and in $\HHH^3$.  For geometrically finite groups, the main facts we need are encapsulated in the following theorem.
Recall that  for a path $\beta \subset \HHH^3$, $[\beta]$ denotes the hyperbolic geodesic joining its endpoints. 

\begin{theorem}\label{floyd} Let $G$ be a finitely generated geometrically finite Kleinian group. There exists a function   $f \co \mathbb N \to  \mathbb N$  such that 
 $f(N) \rightarrow\infty$ as $N\rightarrow\infty$ and such that 
   if $\lambda$ is a $d_G$-geodesic segment in ${\mathcal G}G$ which  lies outside $B_{G}(1; N)$, then both  $j_G(\lambda)$ and   $[j_{G}(\lambda)]$ lie   outside $B_{\HHH}(O;f(N))$ in $\HHH^3$.
\end{theorem} 


Our proof  is based on the following result  of Floyd
 whose proof we recall as it is 
fundamental in what follows.

\begin{prop} [\cite{floyd} p.~216] \label{floydorig} Let $G$ be a finitely generated geometrically finite Kleinian group with generating set $G^*$, and let $O \in \HHH^3$ be a fixed base point.  
Then there exist constants $a,b, k>0$ such that if $G$ contains no parabolics then
\begin{equation}\label{eqn:floyd1}
 b |g|  \leq d_{\HHH }(O, g \cdot  O) \leq a|g| \qquad \forall g \in G, \end{equation}
 while if $G$ contains parabolics then
 \begin{equation}\label{eqn:floyd2}
 2 log |g| -k \leq d_{\HHH }(O, g\cdot  O) \leq a|g| \qquad \forall g \in G.
 \end{equation}
\end{prop}

\begin{proof}
The right hand inequality is easy: $d_{\HHH }(O, g\cdot  O) \leq a|g| $ where $a = \max \{d_{\HHH }(O,e_i\cdot  O)\co  e_i \in G^*\}$.
 
 Now for  the left hand inequality.
 Assume first that $G$ contains no parabolics and thus is convex cocompact. Let $\mathcal D$ be a finite sided fundamental polyhedron  for $G$,  let $\N \subset \HHH^3$ be the hyperbolic  convex hull of the limit set $\Lambda_{G}$, and let $\mathcal D' = \N \cap \mathcal D$. We may as well assume that $O \in \mathcal D'$ and note that by convexity the geodesic from $O$ to $g\cdot  O$ is in $\N$ for all $g \in G$.
 
 Let $d = {\rm diam} \; \mathcal D'$ and let $C = \max \{ |g|: d_{\HHH}(O,g\cdot  O) \leq 1+2d\}$. Let $g \in G$. Divide the geodesic between $O$ to $g\cdot  O$ into intervals of length $1$ (with one shorter interval) and connect each division point to the closest point in $G \cdot O$. This gives the estimate 
 $|g| \leq 1+ Cd_{\HHH}(O,g\cdot  O)$.
 
Now suppose that $G$ contains  parabolics. Choose $\epsilon \leq \epsilon_{\mathcal M}$ such that  $\HH_{\epsilon}$ is an invariant set of disjoint horoballs around the cusps. Let $\mathcal V = \mathcal V(\epsilon) = \N \setminus \cup_{H \in \HH} H$.
Then $\mathcal D'' = \mathcal D' \cap \mathcal V$ is compact and so has finite diameter. By reducing $\epsilon $ if necessary, we can arrange that  the horoballs are small enough so that the geodesic  from $O$ to $e_i \cdot O$ is in $\mathcal V$ for all the generators $e_i \in G^*$.  Define a metric $d'_{\HHH}$ on $\mathcal V$ by setting $d'_{\HHH}(x,y)$ to be  the length of the shortest path in $\mathcal V$ between $x$ and $y$ in the  induced metric on $\mathcal V$.
Then just as above we obtain the estimate  $|g| \leq 1+ Cd'_{\HHH}(O,g \cdot O)$.
Now use  Lemma~\ref{expdist1} which says  that for $H \in \HH$ and 
  points $x,y \in \dd H$, $d'_{\HHH}(x,y) \leq exp \; {d_{\HHH}(x,y)/2}$.
This leads to the left hand inequality in~\eqref{eqn:floyd2}.
\end{proof}


A path $\beta \subset  \HHH$ is a \emph{K-quasi-geodesic} if  for any subsegment $ \alpha \subset \beta $, 
$$ \ell([\alpha])/K- K \leq \ell(\alpha) \leq K \ell([\alpha]) + K$$
where as usual $\ell(\alpha)$ denotes  the hyperbolic length of the path $\alpha$ and  $[\alpha]$ is the hyperbolic geodesic with the same endpoints as $\alpha$.  We use heavily the important fact that  a $K$-quasi-geodesic is at bounded distance from the geodesic with the same endpoints, with constants depending only on   $K$, see for example~\cite{GhH}.

To deal with the thin parts of  the manifold $\HHH^3/G$, we use the following extension of this definition due to McMullen~\cite{ctm-locconn}.
\begin{definition} Let $V \subset \HHH^3$ be a Riemannian manifold. A path $ \beta : [0,1] \to V$ is  an \emph{ambient $K$-quasi-geodesic} if 
$$ \ell(\alpha) \leq K\ell(\gamma) + K$$ for any subsegment $ \alpha  \subset  \beta $ and any path $\gamma \subset  V$ with the same endpoints as $\alpha$.
\end{definition}

  As in Section~\ref{sec:basics},  fix $\epsilon \leq \epsilon_{\mathcal M}$ and let $\HH = \HH_{\epsilon}$ denote the union of the lifts of  the $\epsilon$-thin parts to $\HHH^3$.
We state McMullen's result as applied to $   \mathcal V = \mathcal N \setminus \bigcup_{H \in \HH} H$.

\begin{proposition} [\cite{ctm-locconn} Theorem 8.1]   \label{mcmtracking}  
 Let $ \beta : [0,1] \to \mathcal V$ be an ambient $K$-quasi-geodesic. Then  $\beta$  lies within a bounded distance $R(K)$ of $[\beta] \cup \HH ([\beta])$, where  $\HH ([\beta])$ is the union of those $H \in \HH $ which meet $[\beta]$.
\end{proposition}

\begin{rmk} {\rm Although in this section we only need this result when all elements of $\HH_{\epsilon}$ are horoballs, the result also  holds when  $\HH_{\epsilon}$ contains  Margulis tubes. Moreover the constants involved  depend only on  $\epsilon$ and are independent of the group $G$.} 
\end{rmk}


Here are two ways of constructing ambient quasi-geodesics.

\begin{lemma} \label{floyd2} 
Let  $\gamma \subset \Gr G$ be a path from $1$ to $g$ constructed as in the proof of Proposition~\ref{floydorig}. Then
$\gamma$ is an ambient quasi-geodesic in $\mathcal V$.
\end{lemma}
\begin{proof}
By the construction, $|\gamma| \asymp \ell(\mu)$, where $\mu$ is the shortest path from $O$ to $g \cdot O$ in $\mathcal V$.  
  Since $\ell(j_G(\gamma)) \prec   |\gamma|$, it follows that 
$\gamma$ is also an ambient quasi-geodesic in $\mathcal V$.
\end{proof}

\begin{lemma} \label{floyd1a} Suppose that $G$ is geometrically finite and that $\lambda$ is a geodesic in $(\Gr G,d_G)$. Then $j_G(\lambda)$  is an ambient quasi-geodesic in $\mathcal V= \mathcal N \setminus \bigcup_{H \in \HH} H$.\end{lemma}
\begin{proof}
 If $G$ is convex cocompact, the result follows immediately from
Equation~\eqref{eqn:floyd1}, which says that  
any $d_G$-shortest path  is  a quasi-geodesic in $\HHH^3$.

In the general case, consider any subsegment $\lambda_1 \subset \lambda$  and let $x,y \in \Gr G$ be  its endpoints.
  Let $\mu_1$ be the  shortest path in $\mathcal V$ joining $j_G(x),j_G(y)$.
As in the proof of Proposition~\ref{floydorig}, there exists a path in $\Gr G$ from $x$ to $y$  of length $L$ say, such that $$   \ell(\mu_1) \asymp L.$$
Since $\lambda_1$ is a shortest $d_G$-path from $x$ to $y$ we have $  |\lambda_1|  \leq L$ 
where $ |\lambda_1| $ denotes the length  of the geodesic  $\lambda_1$ in $\Gr G$. Thus
$$\ell(j_G(\lambda_1)) \prec |\lambda_1| \leq L \prec \ell (\mu_1) $$
where the first inequality follows as usual from the right hand inequality of~\eqref{eqn:floyd2}. This  shows that $j_G(\lambda )$ is an ambient quasi-geodesic in $\mathcal V$ as claimed. \end{proof}
 
 \begin{cor} \label{floyd3}Suppose that $G$ is geometrically finite and that $\lambda$ is a geodesic in $(\Gr G,d_G)$.  Then $j_G(\lambda )$    lies within bounded distance of $[j_G(\lambda) ] \cup \HH ([j_G(\lambda) ])$.
 \end{cor}  \begin{proof}
This follows immediately from Proposition~\ref{mcmtracking} and Lemma~\ref{floyd1a}.
 \end{proof}

  \begin{proof} [Proof of Theorem~\ref{floyd}]   
First assume that $G$ contains no parabolics, equivalently, that $G$ is convex cocompact. 
Suppose that $\lambda$ is a $d_G$-geodesic segment in ${\mathcal G}G$ which  lies outside $B_{G}(1; N)$. Then   inequality~\eqref{eqn:floyd1} in Proposition~\ref{floydorig}  gives $d_{\HHH}(O, h \cdot O) \geq b|h| \geq bN$ for all $ h \in j_{G}(\lambda)$. Hence $j_{G}(\lambda)$ lies  outside $B_{\HHH}(O;bN)$.

Equation~\eqref{eqn:floyd1} says that
$j_{G}(\lambda)$ is a $K$-quasi-geodesic in $\HHH^3$ with $K = \max \{b,1/a\}$. Thus $j_{G}(\lambda)$ is at bounded distance from the hyperbolic geodesic $[j_{G}(\lambda)]$ with the same endpoints, with constants depending only on   $K$. Since $j_{G}(\lambda)$ lies  outside $B_{\HHH}(O;bN)$,  it follows that $[j_{G}(\lambda)]$ lies  outside $B_{\HHH}(O;bN - c)$ for some $c>0$ depending only on $a,b$. This completes the proof in the convex cocompact 
 case.

Now suppose that $G$ contains parabolics.  The left hand inequality of~\eqref{eqn:floyd2} 
shows that  $j_{G}(\lambda)$ lies  outside $B_{\HHH}(O;f(N))$ where $f(N) \gadd 2 \log N $.
 Without loss of generality, we may as usual assume that the basepoint $O \in \HHH^3$ lies outside $\bigcup_{H \in \HH} H$.
It remains to show that $\alpha = [j_{G}(\lambda)]$  lies outside some ball  $B_{\HHH}(O;f(N))$ for some $f(N)\to \infty$ as $N \to \infty$.
 
By Lemma~\ref{floyd1a},   $j_G(\lambda )$ is an ambient quasi-geodesic in $\mathcal V$.  
Let $P_1,P_2$ be the entry and exit points of $\alpha$ to some horoball $H \in \HH(\alpha)$. By Proposition~\ref{mcmtracking},  $P_1$ and $P_2$ lie within bounded distance of points on  $j_G(\lambda)$, hence by~\eqref{eqn:floyd2} outside  $B_{\HHH}(O; 2\log N -c')$ for some $c'>0$.
Thus by Lemma~\ref{horoballs}, the segment $[P_1,P_2]$, and hence $\alpha$,   lies outside
$B_{\HHH}(O;  \log N/2 -c'')$ for some $c''>0$. This completes the proof.
\end{proof}


 \section{Existence of  \emph{CT}-maps}
\label{sec:CT}
In this section we state and prove a criterion for the existence of   \emph{CT}-maps, Theorem~\ref{crit1}. Variants will  be used later to prove our main results.

In \cite{mitra-ct}, the existence of  \emph{CT}-maps is discussed in the context of hyperbolic metric spaces.  Suppose that $i: X \rightarrow Y$ is an inclusion of such spaces. A \emph{Cannon-Thurston map} in this context is by definition a continuous extension of $i$ to a map $\hat i \co \hat X \rightarrow \hat Y$, where  for a hyperbolic space $Z$, $\dd Z$ denotes the 
 Gromov boundary and $\hat Z= Z \cup \dd Z $ carries the natural topology obtained by extending the Gromov product to the boundary, see~\cite{BH}.
Lemma 2.1 of \cite{mitra-ct}
asserts that a Cannon-Thurston map exists
if and only if for all $M > 0$ and $x \in X$, there exists $N > 0$ such that if a geodesic $\lambda$ in $X$
lies outside an $N$-ball around $x$ in $X$,  then
any geodesic in $Y$ joining the endpoints of $\lambda$ lies
outside the $M$-ball around $i(x)$ in $Y$.

If now $\rho\co  \G \to G$ is a weakly type preserving isomorphism   of Kleinian groups, then as explained in Section~\ref{sec:basicCT}, the \emph{CT}-map $\hat i \co  \Lambda_{\G} \to \Lambda_G$ is the continuous extension, if it exists, of the embedding 
$i : j_{\G}(\Gr \G) \to \HHH^3$, $ i ( j_{\G}(g)) =  j_{G}\rho(g), g \in \G$.
It is well known that if $H$ is a finitely generated   convex cocompact Kleinian group, then its Cayley graph ${\mathcal G}H$  with the word metric  is a hyperbolic space and the metrics $d_{H}$ on $\mathcal GH$
and the induced hyperbolic metric on $j_{H}(\mathcal GH)$ are equivalent. 
Moreover the limit set $\Lambda_H$ may be naturally identified with the Gromov boundary of ${\mathcal G}H$, see~\cite{GhH}. 
Thus if both groups $\G, G$ are convex cocompact, the above result is a criterion for the existence of the \emph{CT}-map $\hat i : \Lambda_{\G} \to \Lambda_G$. 
The main result of this section is the following theorem which extends this to a criterion which applies without the assumption of cocompactness on either $\G$ or $G$.  
Notice that the hypothesis does not require that the image group $G$ be geometrically finite.

\begin{theorem}\label{crit1} Let $\rho \co \G \to G$ be a weakly type preserving  isomorphism of finitely generated Kleinian groups and suppose that $\G$ is geometrically finite. The \emph{CT}-map  $  \Lambda_{\G} \to \Lambda_G$
 exists if and only if    there exists a non-negative function  $f \co \mathbb N \to \mathbb N$, such that 
 $f(N)\rightarrow\infty$ as $N\rightarrow\infty$, and such that whenever  $\lambda$ is a $d_{\G}$-geodesic segment  lying outside $B_{\G}(1; N)$ in $\Gr \G$,  the  $\HHH^3$-geodesic joining
the endpoints of $i(j_{\G}(\lambda))$ lies outside $B_{\HHH}(O_G; f(N))$ in $\HHH^3$.
\end{theorem}
\begin{proof} Since the result is clearly independent of the choice of basepoints for $\G$ and $G$, for simplicity we take $O_{\G} = O_{G} = O \in \mathbb B$.
As explained in Section~\ref{sec:basicCT}, the existence of the \emph{CT}-map is equivalent to the statement that  $i$ extends  to a continuous map $\hat i \co ( \Lambda_{\G} \cup j_{\G}(\Gr \G), d_{\mathbb E}) \to  (\Lambda_{G} \cup j_{G}(\Gr G), d_{\mathbb E}) $. 

Suppose first that  $i$ extends continuously. For each $N \in \mathbb N$, let $\phi(N) = \sup \{L\}$, where the sup is taken over all $L \geq 0$ with the property that  for all $d_{\G}$-geodesic segments
 $\lambda$  lying outside $B_{\G} (1; N) \subset  \Gr \G$,   the  $\HHH^3$-geodesic   $[i(j_{\G}(\lambda))]$ joining
the endpoints of $i(j_{\G}(\lambda))$ lies outside $B_{\HHH}({O}; L)$.
(Here $\phi(N) =0$ is possible.) Suppose that $ \sup_N \{\phi(N) \} \leq K < \infty$. 
Then we can find a sequence of $d_{\G}$-geodesic segments $\lambda_N \subset \Gr \G$  such that  $\lambda_N$ lies outside the ball $B_{\G} (1; N)$  in $\Gr \G$ while    $\beta_N = [i(j_{\G}(\lambda_N)) ] $   meets the ball 
$B_{\HHH} (O; K+1)$. Thus passing to a subsequence, the endpoints  of $i(j_{\G}(\lambda_N)) $
 converge to distinct points in 
$\Lambda_G$. 
However by Theorem~\ref{floyd}, there exists $f_1(N)$ such that 
the  geodesic
$[j_{\G}(\lambda_N)]$
 lies outside $B_{\HHH}(O; f_1(N))$, and such that $f_1(N) \to \infty$ as $N \to \infty$.
It follows that  after passing to a further subsequence,  the endpoints of $j_{\G}(\lambda_N)$
 converge to the same point in $\Lambda_{\G}$. This contradiction  shows that $\phi(N) \to \infty$ with $N$ and so the criterion is satisfied with $f = \phi$.

Now we show that the condition is sufficient. First we need to define $\hat i \co \Lambda_{\G} \to \Lambda_G$. If $\xi \in \Lambda_{\G}$ is a parabolic point corresponding to a parabolic $p \in \G$, using the hypothesis that $\rho $ is weakly type preserving   we  define
 $\hat i(\xi)$ to be the unique  fixed point of $\rho(p)$.

Now  assume that   $\xi$ is not a parabolic fixed point. 
Note that  $ i: j_{\G}(\Gr \G) \rightarrow \HHH^3$ is proper  (with respect to the hyperbolic metrics), for if not, there exist points $  g_n  \cdot O $ converging to $\Lambda_{\G}$ such that  $\{ i(g_n  \cdot O)\} $ lie in a compact set in $\HHH^3$, which contradicts our hypotheses. 
 By definition, if $\xi \in \Lambda_{\G}$, then there exists a sequence    
 $g_m \in \G$ such that    $j_{\G}(g_m) \rightarrow \xi$ in the Euclidean metric $d_{\mathbb  E}$.  Since $ i$ is proper,  
 $i(j_{\G} (g_m))$ has a subsequence which converges to a point $\eta \in \Lambda_G$.  We want to  define $\hat i (\xi) = \eta$, so we need to see that $\eta$ depends only on $\xi$ and not on the sequence $g_m$. So suppose that $g'_m \in{G}$ and $j_{\G}(g'_m)\rightarrow \xi$ in $(\mathbb B \cup \Chat, d_{\mathbb  E})$, but that  $i(j_{\G} (g'_m))\rightarrow \eta' \in \Lambda_G$ where 
  $\eta \neq \eta'$.

 Let $\lambda_m$ be  the $d_{\G}$-geodesic  joining $g_m$ and $g'_m$. Since $g_m \cdot O\rightarrow \xi$ and $g'_m  \cdot O\rightarrow \xi$, the $\HHH^3$-geodesic     $\alpha_m = [j_{\G}(\lambda_m)]$ joining $j_{\G}(g_m)$ and $j_{\G}(g'_m)$
  lies outside $B_{\HHH}(O, N_m)$,
 where $N_m\rightarrow\infty$ as $m\rightarrow\infty$.  
We claim that $j_{\G}(\lambda_m)$ also  lies outside some ball $B_{\HHH}(O, M_m)$ in $\HHH^3$, where $M_m \to \infty$ as $m \to \infty$.

 Let $ \HH =  \HH_{\epsilon;\G}$ denote the set of lifts to $\HHH^3$ of the thin parts of $\HHH^3/\G$, where $ \epsilon \leq \epsilon_{\mathcal M}$ is chosen so that all elements of $ \HH $ are horoballs. By Corollary~\ref{floyd3}, 
 $j_{\G}(\lambda_m)$ is at uniformly bounded distance to $\alpha_m \cup \HH(\alpha_m)$ where 
$  \HH(\alpha_m) \subset \HH$ is the union of those thin parts   traversed by $\alpha_m$. By Lemma~\ref{horoballs}, if the entry and exit points $P,P'$ of $\alpha_m$ to a component 
$H $ of $  \HH(\alpha_m)$ are at  distance at least $N$ to the base point $O$, then 
the segment $[P,P'] \subset \alpha_m$ is at distance at least $N/4 -c$  from $O$ for some universal $c>0$. It follows that $j_{\G}(\lambda_m)$ is outside a large ball   $B_{\HHH}(O, N'_m)$ where $N'_m \succ N_m$,  unless there is an infinite subsequence of the $\alpha_m$ each of which contains a segment $[P_m,P'_m]$ contained in a horoball $H_m \in    \HH(\alpha_m)$, such that
 $P_m \in B_{\HHH}(O, K)$ for some $K>0$ independent of $m$.  Since there are only finitely many   horoballs which meet $  B_{\HHH}(O, K)$, up to passing to a subsequence we may assume that all the   $ \alpha_m$  pass  through 
 a fixed horoball $H$. Since $\xi$ is not a parabolic point, it is not the basepoint of $H$. Hence by taking $m$ large enough, we can clearly find $g_m \cdot O$ and $g'_m  \cdot O$ close enough to $\xi$ so that $\alpha_m$ does not intersect $H$. This contradiction proves that $j_{\G}(\lambda_m)$  lies outside some ball $B_{\HHH}(O, M_m)$ in $\HHH^3$  as claimed.
 (Note that if $\xi$ is a parabolic fixed point the above discussion fails. For then we can find sequences  $g_m \cdot O, g'_m \cdot O $ which converge to $\xi$ while the $d_{\G}$-geodesic  $\lambda_m$  joining $g_m$ and $g'_m$ is such that $j_{\G}(\lambda_m)$
follows  around the horoball $H \in \HH$ based at $\xi$ and hence always penetrates a  hyperbolic ball  $B_{\HHH}(O, K)$ for fixed $K>0$.)

If  $j_{\G}(\lambda_m)$  lies outside some ball $B_{\HHH}(O, M_m)$, it follows immediately from the inequalities in Proposition~\ref{floydorig} that $\lambda_m$ lies outside some ball $B_{\G}(1, M'_m) \subset \Gr \G$, where $M'_m \to \infty$ as $m \to \infty$.
 On the other hand since
  $\eta \neq \eta'$, there exists $R>0$ such that the $\HHH^3$-geodesic joining  $\eta$ to $\eta'$
has to pass through $B_{\HHH}(O; R)$.  Hence there exist constants $c>0$ and $m_0 \in \mathbb N$ such that for all $m >m_0$,
the $\HHH^3$-geodesic $ [i (j_{\G}(\lambda_m))] $ joining $j_G  (g_m)$ and $j_G  (g'_m)$ 
passes through $B_{\HHH}(O; R+c)$. Since $R+c$ does not depending on the index $m$, 
this  contradicts the hypothesis of the theorem, so  $\eta = \eta'$ and we can define
$\hat i(\xi) = \lim_{m \to \infty} j_G (g_m)$ for any sequence $j_{\G}(g_m) \to  \xi$.
This completes our justification of the definition of the map $\hat i$.

Now we turn to the continuity of $\hat i$. If $\hat i$ is not continuous, there exist  
sequences $x_m,x'_m \in  j_{\G}(\Gr \G)  \cup \Lambda_{\G}$ such that 
$d_{\mathbb E}(x_m,x'_m) \to 0$ but  
so that  $  \hat i(x_m),  \hat i(x'_m)$ converge to distinct points in  $\Lambda_{G}$.
Replacing the points $x_m,x'_m $ by points in $ j_{\G}(\Gr \G)$ if needed, it suffices to  show that 
 for every sequence $g_m \cdot O \in j_{\G}(\Gr \G)$ with $g_m \cdot O  \to \xi$, we have 
$j_G(g_m) \to   \hat i(\xi)$.  This is of course exactly what we have already done, except in the case in which $\xi$ is  the fixed point $p^*$ of a parabolic $p \in \G$.

So suppose that  $x_m = g_m \cdot O \to p^*$  but $i(x_m)$ does not converge to $\rho(p)^*$. 
 Also let $u_m = p^m \cdot O $ so that $ i(u_m) = \rho(p)^m\cdot O  \to  \rho(p^*)$.  
  Let $\lambda_m$ be a $d_{\G}$-geodesic from   $g_m $ to $ p^m $ in $\Gr \G$ and consider the $\HHH^3$-geodesic $\gamma_m =[ j_{\G}(\lambda_m)]$.  If $\ell(\gamma_m)$ is bounded independent of $m$ then by Lemma~\ref{expdist1} $|\lambda_m|_{\G}$ must also be bounded. It follows that
$i(u_m )$ must be at bounded distance from $i(x_m)$ and so $i(x_m) \to \rho(p)^*$ contrary to our assumption.

Otherwise, $\ell(\gamma_m) \to \infty$ as $m \to \infty$.   Let $H$ be the horoball   based at $p^*$. If all but    bounded length initial and final segments of $\gamma_m$ are in $H$, then 
 $x_m$ and $u_m$ are at bounded distance to $H$.  By Lemma~\ref{lemmaV} below, we may assume that 
  $g_m = hp^m$ where $h \in \Stab_{\G} H$. In this case   $i(x_m) = \rho(h)\rho( p)^m \cdot O$
  and clearly $|hp^m|_{\G} \to \infty$ since $\ell(\gamma_m) \to \infty$. Thus $i(x_m)\to \rho(p)^*$ again contrary to our assumption.

 We have thus reduced to the case in which  $\gamma_m$ contains  a  segment $ \gamma'_m$  outside $H$   such that $\ell(\gamma'_m) \to \infty$ with $m$. Let 
  $\lambda'_m = j_{\G}^{-1}\pi^{-1}(\gamma'_m)$ where $\pi$ is the orthogonal projection from $j_{\G}(\lambda_m)$ to $\gamma_m$.
  The projection of $ j_{\G}(\lambda'_m)$ to $\gamma'_m$ is at bounded distance to 
 $\gamma'_m \cup \HH(\gamma'_m)$ and clearly $H \notin \HH(\gamma'_m)$. Thus by the same arguments as above, $ j_{\G}(\lambda'_m)$ lies outside a fixed ball  $B_{\HHH}(O; K)$ for all $m$ and hence
 $\lambda'_m$ lies outside a ball of fixed radius in $\Gr \G$.
 
 On the other hand, the endpoints of $i j_{\G}(\lambda'_m)$ converge to distinct points and so $i j_{\G}(\lambda'_m)$ always meets a fixed ball $B_{\HHH}(O; K')$ for some $K'>0$.
 The sequence $\lambda'_m$ thus violates the hypothesis of the theorem and we have proved 
  that whenever $x_m \to p^*$, $i(x_m)$ converges to $\rho(p)^*$. In view of our previous discussion, this completes the proof of  continuity of $\hat i$.
  \end{proof}

We  immediately deduce our first main result which can be viewed as the convergence of a constant sequence of \emph{CT}-maps.

\begin{theorem} \label{firstresult} Let $\G, G$ be finitely generated geometrically finite groups and let $\phi: \G \to G$ be weakly type preserving isomorphism. Then the \emph{CT}-map
$\hat i: \Lambda_{\G} \to \Lambda_G$ exists.   Moreover if $\xi \in \Lambda_{\G}$  and $g_r \in \G, g_r \cdot O \to \xi$, then $\hat i (\xi) = \lim_{r \to \infty} \rho  (g_r) \cdot O$.
\end{theorem} 
\begin{proof} 
 This follows immediately from Theorems~\ref{crit1} and~\ref{floyd}.
 \end{proof}

Here is the lemma used in the proof of Theorem~\ref{crit1}. 

\begin{lemma}
\label{lemmaV} Let $\G$ be a geometrically finite Kleinian group. There exists $c = c(  \G)>0$ with the following property. 
Let $\lambda $ be a geodesic segment in $\mathcal{G} \Gamma$,  and let  $\pi$ denote projection from $j_{\G}(\lambda)$ to $[j_{\G}(\lambda)]$. Suppose that for some $ g, h \in \G$, the segment of $j_{\G}(\lambda)$ from $j_{\G}(g)$ to $j_{\G}(h)$  projects to an arc contained in a single component $T \in \HH_{\epsilon, \G}$ of the lift to $\HHH^3$ of the $\epsilon$-thin part of $\HHH^3/\G$ for some $\epsilon \leq \epsilon_{\mathcal M}$. Then $g^{-1}h = g^{-1}zgy$ where $z \in Stab_{\G} \, T$ and $d_{\G}(1, y) \leq c$.
\end{lemma}
\begin{proof} 
For $X \in j_{\G}(\lambda)$ let $\Pi(X)$ be the point on $\partial T$ at which the perpendicular from $X$ to  $[j_{\G}(\lambda)]$ meets $\partial T$. 
By Corollary~\ref{floyd3}, there exists $D = D(\G)>0$ so that
 $d_{\HHH}(X, \Pi(X)) \leq D$ for all $ X    \in j_{\G}(\lambda)$.

  Let $\mathcal D $ be the Dirichlet domain for ${\G}$ with centre $O$.  There exists $c>0$ such that any polyhedron which meets the $D$-neighbourhood of $ \mathcal D $
must be of the form $x \mathcal D $ where $ d_{\G}(1,x) \leq c(D)$.
Thus since $g \cdot O  \in  g \mathcal D $ and since $\Pi(g \cdot O)$ is a bounded distance away from $g  \cdot O$, we have
$\Pi(g  \cdot O) \in gx_1  \mathcal D $ where $x_1 \in B_{\G}(1, c(D)) $. Likewise 
 $\Pi(h  \cdot O) \in hx_2  \mathcal D $ with $x_2 \in B_{\G}(1, c(D))$.  Since $\Pi(g  \cdot O), \Pi(h  \cdot O) \in \partial T$, 
there exists $z \in Stab_{\G} \, T$ such that  $z gx_1\mathcal D = hx_2 \mathcal D$ from which we get 
$z gx_1 = hx_2$ and so  $g^{-1}h = g^{-1}zgx_1x_2^{-1}$ which gives the result.
\end{proof}


 \section{The criterion for uniform convergence}
\label{sec:criteria}

In this section we prove our criterion Theorem~\ref{unifcrit1} for the uniform convergence of a sequence of \emph{CT}-maps corresponding to a converging sequence of
representations $\rho_n$.

Let $\G$ be a fixed geometrically finite Kleinian group and suppose that
 $\rho_n\co  \G \rightarrow \PSL$ is a sequence
of discrete faithful weakly type preserving representations converging  algebraically    to $\rho_\infty \co  \G \rightarrow \PSL $. Let $G_n = \rho_n (\G), n = 1,2,\ldots, {\infty}$ and write $\Lambda_n$ for $ \Lambda_{G_n}$. 
To normalize, we embed all the Cayley graphs with the same base point $O = O_{G_n}$
for all $n$ and set $j_{n}(g)=  j_{G_n}(g) = \rho_n(g) \cdot O, g \in \Gr \G$.
Let  $i_n \co j_{\G}(\Gr \G) \to j_{n}(\Gr G_n) $   be the obvious extension  to $\Gr \G$ of the map $j_{\G}(g) \mapsto j_{n}(g), g \in \Gr \G$.

By Theorem~\ref{crit1},   $\hat i_n\co    \Lambda_{\G} \to \Lambda_n$
exists if and only if   there exists a   function  $f_n \co \mathbb N \to \mathbb N$, such that 
 $f_n(N)\rightarrow\infty$ as $N\rightarrow\infty$, and such that whenever  $\lambda$ is a $d_{\G}$-geodesic segment  lying outside $B_{ \G}(1; N)$ in $\Gr \G$,  the  $\HHH^3$-geodesic  $[i_n(j_{\G}(\lambda))]$ lies outside $B_{\HHH }(O;f_n(N))$ in $\HHH^3$. 
Assuming they exist, we shall say that the \emph{CT}-maps   $\hat i_n: \Lambda_{\G} \to  \Lambda_{n}$ converge uniformly (resp. pointwise) to $\hat i_{\infty}$ if they do so as maps from 
$\Lambda_{\G}$ to $\Chat $. 

Before stating the convergence criterion,
we introduce a  property  UEP (Uniform Embedding of Points) of the sequence $(\rho_n)$.

\begin{definition}   Let $\rho_n\co \G \to G_n$ be  a sequence of weakly type preserving  isomorphisms of  Kleinian groups.   Then $(\rho_n)$ is said to satisfy
  UEP   if there
exists a non-negative function  $f \co \mathbb N \to \mathbb N$, with
 $f(N)\rightarrow\infty$ as $N\rightarrow\infty$, such that for all $g
\in \Gamma$, 
 $d_\Gamma (1,g) \geq N$ implies  $d_{\Hyp} (\rho_n(g)\cdot O , O) \geq
f(N)$ for all $n  \in \mathbb N$.  
\end{definition}

Here is an alternative characterisation of UEP.

\begin{lemma} \label{alternative1} For $N \in \mathbb N$, define $$ u_N = \inf \{ t>0:d_{\HHH}(\rho_n(g)\cdot O,O )> t \},$$ where the infimum is taken over all $g \in \G$  with  $ |g|> N$ and all $n$. Then $(\rho_n)$ satisfies UEP if and only if $u_N \to \infty $ as $N \to \infty$.
\end{lemma}
\begin{proof} Clearly $ \ldots u_N \leq u_{N+1}$ for all $N$. If $u_N \to \infty$ as $N \to \infty$, then $(\rho_n)$ satisfies UEP  with $f(N) = u_N$.

Now suppose  there exists $K>0$ so that $u_N \leq K$ for all $N$. Suppose $(\rho_n)$ satisfies UEP with the function $f$. Choose $N$ so that $f(N) > K+1$. From the definition of $u_N$,  there exist $g$,  $n$, such that $|g|>N$  while $j_n(g) \in B_{\HHH}(O, K)$. On the other hand, by UEP $j_n(g)$ is outside  $B_{\HHH}(O, f(N))$. This is impossible.
\end{proof}

\begin{prop} \label{uep}
Suppose that a sequence of  discrete faithful weakly type preserving representations $(\rho_n \co  \G \rightarrow \PSL )$
converges algebraically to  $  \rho_\infty $.  Then $(\rho_n )$ converges strongly  if and only if   it satisfies UEP.
\end{prop}
\begin{proof} Suppose that the sequence of representations $(\rho_n (\Gamma ))$
converges algebraically and
satisfies
UEP with a function $f$.  If the convergence is not strong, then there exists a sequence 
 $(\rho_n(g_{m_n}))$ with $|g_{m_n}| \to \infty$ and $n \to \infty$ which converges in $\PSL $, and hence for which
 $d_{\Hyp} (\rho_n(g_{m_n})\cdot O , O) < M$ for some $M>0$.

Choose $N \in \mathbb N $ such that $f(N) > M$. Then 
$d_\Gamma (1,g) \geq
N$ implies $d_{\Hyp} (\rho_n(g)\cdot O , O) \geq f(N) >M$. So for any $n$,
$d_{\Hyp} (\rho_n(g) \cdot O , O) \leq M$ implies $d_\Gamma (1,g) \leq
N$. 
Since the   ball $B_{\Gamma} (1; N) \subset {\mathcal G}\G$   is finite, we must have  $g_m = g_\ast$ for infinitely many $m$ and some $g_\ast \in
B_{\Gamma} (1; N)$. This contradicts  $|g_m| \to \infty$.

Conversely  if UEP fails, by Lemma~\ref{alternative1} there exists $K>0$ such that with $u_N$ defined as in that lemma,  $u_N \leq K$ for all $N$. 
Thus for all $N$ there exist $g_N \in \G$ and $n_N \in \mathbb N$ such that $d_{\mathbb H}(\rho_{n_N}(g_N)\cdot O , O) \leq K+1$ and $|g_N| > N$. Hence  $(\rho_{n_N}(g_N))$ has a convergent subsequence while $|g_N| \to \infty$, which is impossible by the strong convergence.
\end{proof}

Now we introduce a  further property  UEPP (Uniform Embedding of Pairs of Points) of the sequence $(\rho_n)$.

\begin{definition}\label{crit:unifcrit1} Let $\rho_n\co \G \to G_n$ be  a sequence of weakly type preserving  isomorphisms of  Kleinian groups.  
 Then $(\rho_n)$ satisfies UEPP  if there exists a  function  $f_1 \co \mathbb N \to \mathbb N$, such that 
 $f_1(N)\rightarrow\infty$ as $N\rightarrow\infty$, and such that whenever  $\lambda$ is a $d_{\G}$-geodesic segment  lying outside $B_{\G}(1; N)$ in $\Gr \G$,  the  $\HHH^3$-geodesic   $[j_n( \lambda))]$ lies outside $B_{\HHH}(O; f_1(N))$  for all $n \in \mathbb N$.
\end{definition}

Note that if a sequence of representations $(\rho_n)$ satisfies UEPP, then it automatically satisfies UEP and so by Proposition~\ref{uep}, if it converges algebraically then it also
 converges strongly.
We also remark that the condition of UEPP is just the  statement that the second condition of Theorem~\ref{floyd}  holds uniformly in $n$.    

Here is an alternative characterisation of UEPP, whose proof  is essentially identical to that of Lemma~\ref{alternative}.
\begin{lemma} \label{alternative} For $N \in \mathbb N$, define $$ v_N = \inf \{ t>0: d_{\HHH}([j_n(\lambda)], O) > t \},$$ where the infimum is taken over all $\Gr \G$ geodesics $\lambda$ which are outside $B_{\G}(1, N)$ and all $n$. Then $\rho_n$ satisfies UEPP if and only if $v_N \to \infty $ as $N \to \infty$.
\end{lemma}


Our main criterion for uniform convergence of \emph{CT}-maps is the following:
\begin{theorem}\label{unifcrit1}  Let $\Gamma$ be a geometrically finite Kleinian group  and let $\rho_n: \G \to G_n$ be  weakly type preserving  isomorphisms to Kleinian groups. 
Suppose that $\rho_n$ converges algebraically to a 
representation $\rho_{\infty}$. Then  if   $(\rho_n)$ satisfies UEPP, the \emph{CT}-maps  $\hat i_n\co  \Lambda_{\G} \to \Lambda_{n}$
converge uniformly to $\hat i_{\infty}$. If  $\Gamma$ is non-elementary,  the converse also holds.
\end{theorem}

\begin{rmk}{\rm   The converse result is not needed for the proof  of Theorem~\ref{thm:strong=unif} but we include it for completeness. To see that the converse fails if $\G$ is elementary, consider  the sequence of groups $<A_n>$ where $A_n$ is a single loxodromic converging to a parabolic $A_{\infty}$  in such a way that a subsequence of powers $A_n^{m_n} $ has geometric limit $B$ where $B$ is a parabolic and $<B, A_{\infty}> \ = \mathbb Z^2$.  For the detailed construction  of such an example, see ~\cite{marden-book} Section 4.9.}\end{rmk}

To prove Theorem~\ref{unifcrit1} we need a   lemma which ensures that $d_{\Gamma}$-geodesic paths lying outside a large ball eventually have small visual diameter.

\begin{lemma} \label{tail}
Let $\rho\co \G \to G_n$ be  weakly type preserving  isomorphisms of finitely generated Kleinian groups.
  Suppose that $G_n$ converges algebraically to a  group $G_{\infty}$, normalised as above. Suppose that  $(\rho_n)$ satisfies UEPP.
 Then there exists $f_2 \co \mathbb N \to \mathbb N$ such that  $f_2(N) \to \infty$ as $N \to \infty$, and such that for all $g \in \mathcal G\G$ lying outside $B_{ {\Gamma}}(1;N)$  and $m, n \geq f_2(N)$, the $\HHH^3$-geodesic  
 $[j_m(g), j_n(g)]$ lies outside $B_{\HHH }(O;f_2(N))$.
\end{lemma}
\begin{proof}  
Given $N \in \mathbb  N$, by the algebraic convergence we can choose $N_0$ such that for all $m, n \geq N_0$ and $g \in B_{{\Gamma}}(1;N+1)$ we have
 $d_{\Hyp}(j_m(g),j_n(g)) \leq 1$. Now, let $h \in \Gamma $ be such that $d_\Gamma(h,1) \geq N+1$.  For $g,g' \in \Gr \G$, denote by 
 $ [g,g']_{\G}$   a $d_{\G}$-geodesic segment in $\Gr G$ from $g$ to $g'$.
Let $h_0 \in [1,h]_{\G}$
such that $d_\Gamma (1, h_0) = N$. By hypothesis, the $\Hyp^3$-geodesic segments $\gamma_m = [j_m([h_0,h]_{\G})]$  and $\gamma_n = [j_n([h_0,h]_{\G})]$ lie  outside  $B_{\HHH}(O; f_1(N))$. Let $\delta$ be a hyperbolicity constant for $\Hyp^3$, in the sense that any side of a triangle is contained in a $\delta$-neighbourhood of the other two.  Then the $\Hyp^3$-geodesic segment  $[j_m(h),j_n(h)]$ is within distance $2\delta$  of  $\gamma_m \cup \gamma_n \cup [j_m(h_0),j_n(h_0)]$ and hence lies outside the $(  f_1(N) - 1 - 2\delta)$-ball around 
$O \in \Hyp^3$. Choosing $f_2(N) = \max \{N_0, f_1(N) - 1 - 2\delta \}$ gives the result.
\end{proof}

\begin{proof}[Proof of Theorem~\ref{unifcrit1}] 
Given $\xi \in \Lambda_{\Gamma}$, choose $g_r \in \G, g_r \cdot O \to \xi$. Since $\G$ acts properly discontinuously on $\HHH^3$,  $|g_r| \to \infty$ as $r \to \infty$. By Theorem~\ref{firstresult}, $\hat i_n(\xi) = \lim_{r \to \infty} \rho_n(g_r \cdot O)$ for $n = 1,2,\ldots, \infty$. We first show that  if  $(\rho_n)$ satisfies UEPP, then
 $ \lim_{r \to \infty} \rho_n(g_r) \cdot O$ converges to    $ \lim_{r \to \infty} \rho_{\infty}(g_r) \cdot O$ uniformly in $\xi$ as $n \to \infty$.
 
 By UEPP, given $N \in \mathbb N$, there exists $f_1(N)$ such that  for all $n$ and $ \xi$, 
 $i_n(g_r) = \rho_n(g_r) \cdot O$ is outside $B_{\HHH}(O;f_1(N))$ whenever $|g_r| > N$. Thus by Lemma~\ref{tail}, the $\HHH^3$-geodesic $[i_n(g_r),i_m(g_r)]$ is outside $B_{\HHH^3}(O;f_2(N))$
 whenever $|g_r| \geq N$ and $n,m \geq f_2(N)$, where $f_2(N) $ is determined from $f_1(N)$ as in that lemma. Thus working in the ball model with  $d_{\mathbb E}$ denoting Euclidean distance, from Lemma~\ref{easy} we have $d_{\mathbb E}(i_n(g_r),i_m(g_r)) \prec e^{-f_2(N)}$ whenever $|g_r| > N$ and $n,m \geq f_2(N)$.

 Now $i_n(g_r) \to \hat i_n(\xi) $ as $r \to \infty$. Moreover  by  UEPP and Lemma~\ref{easy}, we have $d_{\mathbb E}(i_n(g_r),i_n(g_s)) \prec e^{-f_1(N)}$  whenever $|g_r| , |g_s|> N$.  Thus $d_{\mathbb E}(i_n(g_r), \hat i_n(\xi)) \prec e^{-f_1(N)}$
 whenever $|g_r|> N$.
  So      $d_{\mathbb E}(\hat i_n(\xi),\hat i_m(\xi))  \prec (e^{-f_2(N)} + e^{-f_1(N)})$ which gives uniform convergence.

To prove the converse, suppose the convergence is uniform and that $\Gamma$ is non-elementary. Uniform convergence implies diagonal convergence and hence that the limit sets of the geometric and algebraic limits are the same. Now  Theorem~\ref{evans}, together with the assumption that $\G$ is  non-elementary, gives that the convergence is strong. By Proposition~\ref{uep}, this implies UEP.

Suppose that UEPP fails. By Lemma~\ref{alternative}, the sequence $v_N = \inf \{ t>0: d([j_n(\lambda)], O) > t \}$ is bounded above by $K$ say, so for all $N$ there exists a $\Gr \G$ geodesic $\lambda_N$ and $n_N \in \mathbb N$ such that $\lambda_N$ is outside 
$B_{\G}(1, N)$ while $ [j_{n_N}(\lambda_N)]$ intersects the ball $B_{\HHH}(O, K)$.

We claim that $n_N \to \infty$ as $N \to \infty$. If not, there exists $L>0$ such that 
$n_N \leq L$ for all $N$.  By Theorem~\ref{floyd},  for each $r$ there is a function $f_r\co \mathbb N \to \mathbb N$ with $f_r (N) \to \infty$ with $N$, and such that   if $\lambda$ is outside $B_{\G}(1,N)$ then
$[j_r(\lambda)]$ is outside $B_{\HHH}(O, f_r(N))$.
Thus if $N$ is large enough that $\min \{f_1(N),f_2(N),\ldots, f_L(N) \} >K$ then $[j_{n_N}(\lambda_N)]$ is outside $B_{\G}(1,K)$, contrary to the choice of $\lambda_N$.

Thus the sequence $n_N $ is unbounded, so that we can   choose 
 a sequence $n_{r} \to \infty$ and  $\Gr \G$ geodesics  $\lambda_{r}$ such that $\lambda_r$ is outside 
$B_{\G}(1, r)$ while  $[j_{n_{r}}(\lambda_{r})]$ meets $B_{\G}(1,K)$ for all $r$.

  Suppose that $ \lambda_r$ has endpoints $g_r,h_r \in {\mathcal G} \G$.
  By Theorem~\ref{floyd}, after passing to a subsequence we may assume that the points $g_r \cdot O$ and $h_r\cdot O$ converge to the same point $\xi \in \Lambda_{\Gamma}$.
  It follows from Lemma~\ref{easy} combined with UEP and the uniform convergence that 
$i_{n_r}(g_r ) $ and $i_{n_r}(h_r )$ both limit on $\hat i_{\infty}(\xi) \in \Lambda_{\infty}$, contradicting the fact that $[j_{n_{r}}(\lambda_{r})]$
meets $B_{\mathbb H}(O,  K)$ for all $r$.
\end{proof}

\begin{proof}[Proof of Theorem~\ref{lambdalemma}] 
Theorem~\ref{thm:strong=unif}  gives an alternative proof of Theorem~\ref{lambdalemma}.
With the notation of the statement in the Introduction, by Theorem~\ref{firstresult}, if $\G$ is a finitely generated Fuchsian group and $\rho \co \G \to G_z$ is a type preserving isomorphism to a  quasi-Fuchsian group $G_z$, the natural   map $i_z \co \Lambda_{\G}^+ \to \Lambda_{G}^+$ extends to a continuous map $\hat i_z \co \Lambda_{\G} \to \Lambda_{G}$.

Now replace the parameter $z \in A$  by a sequence $(z_n) \to z_\infty$ and write $G_n$ for $G_{z_n}$ etc.
Suppose   that  $\G, G_n, G_\infty$ are all quasifuchsian. 
If $(G_n)$ converges algebraically  to $G_\infty$,  it automatically converges strongly. Hence by Theorem~\ref{thm:strong=unif}   the $CT$-maps $\hat i_n : \Lambda_{\G} \to  \Lambda_{G_n} $ converge uniformly to  $\hat i_\infty$. Uniform convergence gives diagonal convergence. Since this argument applies to any  sequence $z_n \in A$ with limit $z_\infty \in A$, we easily deduce joint continuity of the map $\hat i_z$. 

Let $\xi \in \Lambda_{\G}$  and pick a sequence of attractive fixed points $\xi_m = g_m^+ \to \xi$ where $g_m \in \G$ are hyperbolic. The maps $ z \mapsto  \rho_z(g_m)^+ = \hat i_z(\xi_m)$ are holomorphic for each $m$ and by our result  $\hat i_z(\xi_m) \to \hat i_z(\xi ) $ for each $z$.  Moreover  the family of maps $ z \mapsto \hat i_z(\xi_m)$ is uniformly bounded and hence normal. It follows that  $ z \mapsto \hat i_z(\xi)$ is holomorphic as claimed. 
\end{proof} 
 

\section{Strong convergence}  \label{sec:strong}

In this section we prove Theorem~\ref{thm:strong=unif},  that if a geometrically finite group $G_{\infty}$ is a strong limit of a  sequence $G_n$, then the corresponding $CT$-maps converge uniformly.  To do this, it is sufficient in view of Theorem~\ref{unifcrit1} to check the criterion UEPP.
The main ingredient is the following uniform bound on  the diameters of the thick parts of the convex cores. Note that as long as the geometric limit is geometrically finite, the hypothesis only requires algebraic rather than strong convergence; this will be important when we come to the proof of Theorem~\ref{thm:alg=ptwise}.

  \begin{proposition} \label{prop:diameters} Suppose given a sequence of geometrically finite Kleinian groups $G_n = \rho_n(\G)$ 
  where $\Gamma$ is   geometrically finite and the representations $\rho_n$ are faithful and weakly type preserving. 
Suppose that the sequence $(G_n)$ converges geometrically to a geometrically finite group $H$. Then the algebraic limit $G_{\infty}$ is geometrically finite. Moreover the thick parts of the convex cores have uniformly bounded diameters.
\end{proposition}
\begin{proof} For the first statement, see~\cite{marden-book} Theorem 4.6.1.
By the same result, the limit sets $\Lambda_n$ converge  to $\Lambda_H$ in the Hausdorff topology and  the ordinary sets $\Omega_n$ converge  to $\Omega_H$ in the sense of Carath{\'e}odory.

Let $\mathcal D_H$ be a fundamental domain for $H$ and let $ \mathcal N_H$ be the hyperbolic convex hull of the limit set $\Lambda_H$. Denote by $\mathcal V_H=\mathcal V_{\epsilon,H}$ the $\epsilon$-thick part of $ \mathcal N_H$ relative to $H$ and define $ \mathcal D_n, \mathcal N_n, \mathcal V_n$ similarly.

Since $H$ is geometrically finite, we can find a hyperbolic ball $B_r = B_{\HHH}(O;r) \subset \HHH^3$ which contains $ \mathcal V_H \cap \mathcal D_H$.
By geometric convergence we have uniform convergence of $\mathcal D_n$ to $\mathcal D_H$ inside 
$B_r$. Moreover the Hausdorff convergence of $\Lambda_n$    to $\Lambda_H$
implies that  inside $B_r$, $ \mathcal N_n$ is eventually contained in a bounded neighbourhood  of $ \mathcal N_H$. This does not however automatically give a uniform bound on the diameters of $\mathcal N_n$, as it says nothing about what happens far outside $B_r$.

To understand the problem outside $B_r$, consider the following toy example.  Let   $\mathcal Z = \{(x,y)\co  x \in [1,\infty), 0 \leq y \leq 1/x   \}  \subset   \mathbb R^2$.
Let $\mathcal Z_n$ be the part of $\mathcal Z $  with $  x \in [0, 2n]$. Let  $C_n \subset \mathbb R^2$ be the disk of radius $n$ and centre $2n \in \mathbb R$ and let $Q_n =   \mathcal Z_n \cup C_n$. Then the sets $Q_n$ converge  uniformly on compact sets in the  Hausdorff metric to $  \mathcal Z$ and  $\mathcal Z$ has finite area, but   ${\rm diam}  \; Q_n \to \infty$.

To resolve this problem it we shall prove the following claim: there exists $A>0$ such 
that
 if $ \mathcal V_n  \cap \mathcal D_n$ contains points
outside $B_{r+2+A}$, then there are points of $ \mathcal V_n  \cap \mathcal D_n $  in the shell between $B_{r+1}$ and $B_{r+2}$.

Suppose the claim holds.  Suppose that (up to a subsequence) there are points $X_n \in \mathcal V_n  \cap \mathcal D_n$ with $d_{\HHH}(O, X_n) >r+2+A$. Use the claim to choose points 
$Y_n \in \mathcal V_n \cap \mathcal D_n $ with $r+1 \leq d_{\HHH}(O, Y_n) \leq r+ 2 $. By compactness we may  assume $Y_n \to Y \in \mathcal N_H$, and by the geometric convergence of $\mathcal D_n$ to $\mathcal D_H$ inside $B_{r+2 }$, we have $Y \in \overline{ \mathcal D}_H$. Since $Y$ is outside $B_r$ it must be in the complement of $\mathcal V_H \cap \mathcal D_H$, that is, in  the thin part of 
$ \mathcal N_H \cap \mathcal D_H$. So there exist $h \in H$ and a neighbourhood $U$ of $Y$ such that  such that $d_{\HHH}(hY', Y') < \epsilon$ for all $Y' \in U$.
Choose $g_n \in \G$ with $\rho_n(g_n) \to h$. We have  $d_{\HHH}(\rho_n(g_n)Y', Y') < \epsilon$ for all large enough $n$
and hence  $d_{\HHH}(\rho_n(g_n)Y_n, Y_n) < \epsilon$ as $n \to \infty$. This contradicts the choice of $Y_n$.
We conclude that  eventually 
$\mathcal V_n \cap \mathcal D_n$ is contained in $B_{r+2+A}$ so that the sequence of diameters $ \mathcal V_n \cap \mathcal D_n$ is uniformly bounded above.

 Now we prove the claim, which can be seen as a very rudimentary form of Canary's filling theorem~\cite{canary}.  Since all groups $G_n$ are isomorphic and geometrically finite, the boundaries  $  \partial \mathcal N_n/G_n$ are homeomorphic and  the thick parts of $ \partial \mathcal N_n/G_n$
have uniformly bounded diameter $A$ say (depending only on the maximum genus of the hyperbolic components of $  \partial \mathcal N_{\Gamma}/{\Gamma}$).  
  If there are points of $\mathcal V_n \cap \mathcal D_n $ outside $B_{r+2+A}$, then there are also points of $\partial \mathcal V_n \cap \mathcal D_n$ outside $B_{r+2+A}$. Then any component $S_n \subset \partial \mathcal N_n/G_n$ containing such points  is  outside $B_{r+2}$.
Since $ \partial \mathcal N_n/G_n$ has finitely many components,   passing to a subsequence, we may fix one such component which is homeomorphic to some fixed hyperbolic surface $S$, and whose thick part lies outside $B_{r+2}$.  (Since we arranged that $\mathcal V_n \cap \mathcal D_n \subset B_r$), we do not have to worry about components of $ \partial \mathcal N_n/G_n$
on the boundary of horoballs or tubes.)
 
Fix a lift $\tilde S \subset \dd \mathcal N_n$ and let  $W \subset \Chat$ denote the component of the regular set $\Omega_{\Gamma}$ of $\G$ corresponding to $\tilde S$. Let $K = \Stab \; W  \subset \Gamma$. Since $K \subset \Gamma$ is  non-elementary (because $\tilde S/K$  is a hyperbolic component of $ \partial \mathcal N_{\G}/\G$), we can choose a pair 
of  non-commuting elements  $\alpha, \beta \in K$, both of which are non-trivial in $  \Gamma$ and hence in $  H$.  (For the fact that $\Gamma$ injects into $H$, see Lemma 4.4.1 in~\cite{marden-book}.) Since by geometric finiteness only finitely many elements of $\Gamma$ are parabolic or short in $H$, we may assume in addition that neither $\alpha$ nor $\beta$ is parabolic or the core of a large Margulis tube in $H$.

Set $ M_n= \HHH^3/G_n$ and $ M_H = \HHH^3/H$ and let $O_n^*, O_H^*$ be the projections of the basepoint $O \in \HHH^3$ to $M_n, M_H$  respectively. 
Let $d_n, d_H$ denote distance in 
$M_n, M_H$ respectively.
Note that no pair of loops    in the homotopy classes $[\alpha], [\beta]$  of $  \alpha, \beta   $ can be contained in the same component of a thin part of $ M_n $, since the fundamental group of any such component  is abelian.
Choose loops $  \alpha_n \in [\alpha],   \beta_n \in [\beta]$ on $ S \subset \partial \mathcal N_n/G_n$.  
 By construction the distance from each of  $  \alpha_n,   \beta_n  $  to $O_n^* $ is at least  $ {r+2}$. 
On the other hand, in   $M_H  $, the geodesic representatives   $\rho_{\infty}(\alpha )^*,\rho_{\infty}(\beta)^*$  are contained inside $\mathcal N_H/H$ and hence, since by construction neither is parabolic in $H$ and additionally neither is in the thin part of $M_H  $, $d_H(\rho_{\infty}(\alpha )^*,O^*) \leq r$ and similarly for $\beta$. By algebraic convergence, we can find lifts to $\HHH^3$ of   the geodesic representatives   $  \rho_n(\alpha)^*,  \rho_n(\beta)^*$  of $ [\alpha],    [\beta]$ in $M_n$    near  to corresponding  $\rho_{\infty}(\alpha )^*,\rho_{\infty}(\beta)^*$ in $\HHH^3$. Hence $d_n(\rho_{n}(\alpha )^*,O_n^*) \leq r+1$  for sufficiently large $n$, and similarly for $\beta$.

Choose homotopies  $H^n_{\alpha}, H^n_{\beta} $ 
 between  
$ \rho_n(\alpha)^*$  and $  \alpha_n$ and   between $ \rho_n(\beta)^*$  and $\beta_n$.
More precisely, let  $H^n_{\alpha}$  be the image of a continuous family of maps $f_t \co [0,1]  \to M_n, t \in [0,1] $ with $f_0([0,1] ) =   \rho_n(\alpha)^*$ and  $f_1([0,1] ) =  \alpha_n$, and similarly for $H^n_{\beta}$.

Let $\dd B(O_n^*;{R})$ denote the boundary of the ball of radius $R$ centre $O_n^*$ in $M_n$. By construction, $  \alpha_n,   \beta_n  $ lie on a component of $\partial \mathcal N_n/G_n$
outside $B(O_n^*;{r+2})$, so  $H^n_{\alpha}, H^n_{\beta}$ must both intersect the shell between $B(O_n^*;{r+1})$ and $B(O_n^*;{r+2})$, and in particular the surface $\dd B(O_n^*;{r+3/2})$. (Note that $\dd B(O_n^*;{R})$ is just the projection of the boundary of the ball of radius $R$ in $\HHH^3$ and hence an immersed $2$-manifold in $M_n$.) Hence we can find paths $\alpha_1 \in [\alpha], \beta_1 \in [\beta]$ in $ \dd B(O_n^*;{r+3/2})$. In detail, by adjusting $H^n_{\alpha}$ slightly if needed, we can arrange that $\dd B(O_n^*;{r+3/2})$ and $H^n_{\alpha}$ are transverse, so that the intersection $\dd B(O_n^*;{r+3/2}) \cap H^n_{\alpha}$ consists of a finite number of closed loops. Adjust $H^n_{\alpha}$ to `push off' any components which are homotopically trivial. Since $H^n_{\alpha}$ is a cylinder whose two boundary components  are separated by $\dd B(O_n^*;{r+3/2})$,   $\dd B(O_n^*;{r+3/2}) \cap H^n_{\alpha}$ must contain at least one homotopically non-trivial component which we take to be $\alpha_1$. The construction of $\beta_1$ is similar.

As discussed above, $\alpha_1$ and $ \beta_1 $ cannot both be contained in the thin part of $\mathcal N_n/G_n $. Lifting to $\HHH^3$, this produces the points $Y_n$ as required.

Thus $\mathcal V_n \cap \mathcal D_n$ is eventually contained in a compact ball $B_r$, which immediately gives the required bound.

\end{proof}

It is now easy to prove Theorem~\ref{thm:strong=unif}.

 \begin{proof}[Proof of Theorem~\ref{thm:strong=unif}] By Theorem~\ref{unifcrit1}  it is  sufficient to check that the sequence $(\rho_n)$ satisfies UEPP.
   In other words, we have to check that  the second condition   of Theorem~\ref{floyd}   holds uniformly
 in all the groups $G_n$.  To do this, we go through that proof of Floyd's result Proposition~\ref{floydorig}  with this requirement in mind.
 (Notice that $\rho_{\infty}$ is automatically weakly type preserving.)

 First consider the right hand inequality of~\eqref{eqn:floyd1}: 
 $d_{\HHH}(O, g \cdot O) \leq a|g| $ where $a = \max_i d(O,e_i \cdot O)$ and $\{e_1, \ldots, e_k\}$ is a finite set of generators for $\G$.
 Algebraic convergence implies that $\rho_n(e_i)$ converges to $\rho_{\infty}(e_i)$ for each $i$, so we have $  \max_i d(O, \rho_n(e_i) \cdot O) \leq A$  for some uniform $A$ independent of $n$. Thus the right hand side of~\eqref{eqn:floyd1} holds uniformly, precisely: 
 $$d_{\HHH}(O, \rho_n(g)\cdot O) \leq A|\rho_n(g)|  = A|g| \  {\rm for \ all}  \ g \in \G   \ {\rm and} \  n \in \mathbb N.$$

We now need to prove the left hand inequality. The constants involved are the diameter $d$ of the truncated fundamental domain $  \mathcal D \cap \mathcal V$, and 
the constant $C = \max \{ |g|: d_{\HHH}(O,g \cdot O) \leq 1+2d\}$ in the proof of  Proposition~\ref{floydorig}.

The diameters of the truncated fundamental domain $\mathcal D_n \cap \mathcal V_n$ are uniformly bounded by Proposition~\ref{prop:diameters}. 
That we can choose the constant $C_n = \max \{ |g|: d_{\HHH}(O,g \cdot O) \leq 1+2d\}$ with a uniform bound $ C_n \leq C$ for all $n$  follows immediately from UEP which has already been proved in Proposition~\ref{uep}. 

The proof now follows exactly as in Theorem~\ref{floyd}, using in addition the distortion Lemmas~\ref{expdist2} and~\ref{tubes}  for Margulis tubes. 
\end{proof}

\section{Algebraic limits}   \label{sec:algebraic}

 \subsection{Pointwise convergence criterion}
\label{sec:ptwise} 

In this section we prove Theorem~\ref{thm:alg=ptwise}. 
We will use conditions similar to UEP and UEPP,   relaxed so as  to allow for dependence on the limit point $\xi$.
We call these new conditions $EP(\xi)$ (Embedding of Points)  and $EPP(\xi)$ (Embedding of Pairs of Points).

Let $G$ be a  geometrically finite Kleinian group with  generators  $G^* = \{e_1, \ldots, e_k \}$, and let $\xi \in \Lambda_{G}$.  
Suppose given  a sequence 
  $e_{i_1},e_{i_2} \ldots  \in G^*$  such that setting $g_r = e_{i_1} e_{i_2} \ldots e_{i_r}$
we have
  $g_n \cdot O  \to \xi$ in $\mathbb H^3 \cup \partial \mathbb H^3$. 
   We denote the infinite path  joining vertices  $1, g_1,g_2, g_3, \ldots   $ in $\Gr G$  by $[1,\xi )$ and    write  $\hat j_{G}(e_{i_r}e_{i_2} \ldots e_{i_s} )  $ for the piecewise geodesic path in $\HHH^3$ joining in order the points  $j_{G}({g_r}), j_{G}({g_{r+1}}),  \ldots, j_{G}({g_s}) $.

\begin{definition}\label{critep}  Let $\Gamma$ be a fixed finitely generated
 Kleinian group and $\rho_n\co \Gamma  \to G_n$ be a sequence
 of  isomorphisms to Kleinian groups $G_n$.  Let $\xi \in \Lambda_{\G}$ and let $(g_r) = [1,\xi)$ be any infinite path as above.  
 \begin{enumerate} \item The pair $((\rho_n),[1,\xi))$ is said to satisfy
  $EP(\xi)$  if there
exist functions  $f_{\xi}  \co \mathbb N \to \mathbb N$ and $M_{\xi} \co \mathbb N \to \mathbb N$, with
 $f_{\xi}(N) \rightarrow\infty$ as $N\rightarrow\infty$, such that for all $g
\in [1,\xi)$, 
 $d_\Gamma (1,g) \geq N$ implies  $d_{\Hyp} (\rho_n(g)\cdot O , O) \geq
f_{\xi}(N)$ for all $ n \geq M_{\xi}(N)$. 
 
\item The pair $((\rho_n),[1,\xi))$ satisfies $EPP(\xi)$ if there exists a function  $f_{1,\xi}(N )\co \mathbb N \to \mathbb N$ such that $f_{1,\xi}(N)\rightarrow \infty$ as $N\rightarrow \infty$
such that for any  subsegment $e_{i_r}  e_{i_{r+1}} \ldots e_{i_s}$ of  $[1,\xi ) $ lying outside $B_{\G}(1;N)$ in $\mathcal G \G$,  
the geodesic $[j_n (g_{i_r}), j_n(g_{i_s})]$ lies outside $B_{\HHH} (O; f_{1,\xi}(N))$  for all $n \geq M_{\xi}( N)$, where $M_{\xi}$ is as in (1). \end{enumerate}
\end{definition}

  \begin{remark}
{\rm   Although \emph{a priori} the   conditions $EP(\xi), EPP(\xi)$  depend on the choice of sequence $e_{i_1}  e_{i_2} \ldots$,  it will follow from our proof below that  (in the case of  geometrically finite limits)  provided we choose suitable paths $[1,\xi)$, they depend only on $\xi$. To do this we will restrict the class of paths used, so that  $[1,\xi)$ is quasi-geodesic and satisfies an additional hypothesis which ensures that it tracks shortest Euclidean paths across the boundaries of rank $2$-horoballs. We call such paths \emph{standard}.   It is not completely obvious that such paths $[1,\xi)$ exist; we   prove this   in Proposition~\ref{quasiwordpath} below.  We will show that  for a sequence $(\rho_n)$ of representations as in the hypotheses of Theorem~\ref{thm:alg=ptwise},
the   conditions $EP(\xi)$ and $ EPP(\xi)$  hold for any standard quasi-geodesic $[1, \xi)$.  }
\end{remark}

Although clearly, $EPP(\xi)$ implies $EP(\xi)$,  we shall first prove $EP(\xi)$ and deduce $EPP(\xi)$.

\begin{theorem} \label{ptwisecrit} Suppose  that   $\rho_n : G \rightarrow \PSL$ is a sequence
of discrete faithful representations converging algebraically to $\rho_\infty \co G \rightarrow \PSL $, and suppose the corresponding CT maps $\hat i_n \co  \Lambda_{\G} \to \Lambda_{G_n}$ exist, $n = 1,2 \ldots,\infty$.
Let $\xi \in \Lambda_{\G}$.  Then
$\hat i_n (\xi) \to \hat i_{\infty}(\xi) $  as $n \to \infty$
 if  $((\rho_n);\xi)$ satisfies $EPP(\xi)$.
\end{theorem}
\begin{proof} The proof is the same as that of  Theorem~\ref{unifcrit1}.  Lemma~\ref{tail} works just as before with  the condition $EPP(\xi)$ replacing  $UEPP$.  Notice that in both of these proofs, it is sufficient to require that the conditions (1) and (2) hold only for all $ n \geq M(N)$.\end{proof}

Contrary to the case of strong convergence, the converse to Theorem~\ref{ptwisecrit}, that if $\hat i_n (\xi) \to \hat i_{\infty} (\xi)$ 
  then $((\rho_n);\xi)$ satisfies $EPP(\xi)$, is false. In fact if $\xi$ is the fixed point of a loxodromic element $p$ which becomes parabolic so as to give rise to a $\mathbb Z^2$- cusp in the geometric limit, then $\hat i_n (\xi) \to \hat i_{\infty}(\xi) $ by the algebraic convergence but $EP(\xi)$ fails, because we have a sequence $\rho_n(p^{m_n}) $ where $m_n \to \infty$ with $d_{\HHH}(\rho_n(p^{m_n}) \cdot O, O)$ remaining bounded. This is discussed further below.

\subsection{Standard quasi-geodesics}
Recall from Proposition~\ref{mcmtracking} that if $G$ is geometrically finite, then any ambient quasi-geodesic $\beta$ in the thick part of the convex core $\mathcal V_{G}$ is a bounded distance from $[\beta] \cup \HH ([\beta])$.
In what follows, we would like to assert further that any two ambient quasi-geodesics in $\mathcal V_{G}$ are a bounded distance apart, with constants depending only on $G$. However this may not be true if  $G$ contains rank  two cusps. Indeed let $H \subset \HHH^3$ be a horoball   corresponding to a such a cusp and let $x,y \in \partial H$. Since the induced metric on $  \partial H$ is Euclidean, there can be no  bound on the distance between  quasi-geodesic paths from $x$ to $y$ independent of $d_{\HHH}(x,y)$. To deal with this, if $H$ is a rank  two  horoball and $x,y \in \partial H$, we say a path from $x$ to $y$ is \emph{D-standard}  (resp. \emph{standard}) if it is within bounded distance $D$ (resp. bounded distance depending only on $G$)
 of the Euclidean geodesic from $x$ to $y$ on $ \partial H$. 
 
 Let $\beta$  be an ambient quasi-geodesic $\beta$  in $\mathcal V = \mathcal V_G$.
 For a component $H \in \HH([\beta])$, let  $\beta_H \subset \beta$ be the 
segment which projects to $[\beta] \cap H$ and 
 for $x \in \beta_H$  denote by 
 $\Pi_H(x) \in \dd H$ the point where the perpendicular from $x$ to $[\beta]$ meets $\dd H$.
We say 
 $\beta$  is  \emph{standard} if  
 for any component $H$ of $\HH([\beta])$ corresponding to a rank  two  cusp, the path 
  $\Pi_H(\beta)$ given by  $x \mapsto \Pi_H(x);  x \in  \beta_H $ across $\partial H$  joining the initial and final points  of $[\beta] \cap H$  is   standard. 
  Finally we call a   path $\lambda \subset \Gr G$  \emph{standard}  if  its image 
  $j_{G}(\lambda) $  in $\mathcal V_{G}$ is standard. 
  
 

  \begin{lemma}  \label{ambientdist} Let $K>0$.  The image of a standard $K$-quasi-geodesic in $(\Gr G,d_G)$ is a standard ambient quasi-geodesic in $\mathcal V_{G}$. Moreover
any two standard $K$-quasi-geodesics in  $(\Gr G,d_G)$ with the same endpoints are a bounded distance apart in $\HHH^3$ with a constant which depends only on $K$ and $(G, G^*)$.
\end{lemma}
\begin{proof} 
Let $\gamma $ be a standard  $K$-quasi-geodesic  in $\Gr G$. That $j_{\G}(\gamma )$ is an ambient $K'$-quasi-geodesic in $\mathcal V_{G}$ for suitable $K'= K'(K,G)$ follows from the same argument as in Lemma~\ref{floyd1a}. By definition its image in $\mathcal V_{G}$ is standard.




Now let $\beta, \beta'$ be two standard ambient  quasi-geodesics in  $\mathcal V$ with the same endpoints so that 
 $  [\beta] = [\beta']$ is the   $\HHH^3$-geodesic joining their common endpoints.  
 Adjust the  constant $\e$ so that the only components of $M_{thin}(\e)$ of $\HHH^3/G$
are horoballs. 
By Proposition~\ref{mcmtracking} both   $\beta, \beta'$   are within bounded distance of $[\beta] \cup \HH([\beta])$  with constants which depend  only on $K$ and $(G, G^*)$.

We have to show that the projections of $\beta, \beta'$ onto $ [\beta]   \cup \HH([\beta])$ are bounded distance apart. 
This is certainly true of the projections onto segments of $ [\beta]$ outside  $\HH( [\beta])$.
Suppose $H \in \HH( [\beta])$ is a horoball corresponding to a parabolic $p \in G$.  
Let $\eta$ be the Euclidean shortest path 
in $  \partial H \cap \mathcal V$ joining the entry and exit points of $ [\beta]$ to $H$. If $H \in \HH( [\beta])$ corresponds to a rank one parabolic, then since  $G$ is geometrically finite, $\partial H \cap \mathcal V$ is a strip of bounded width containing $\eta$ and which also contains the paths  $\Pi_H(\beta), \Pi_H(\beta')$, so the result is clear. If $H  $ corresponds to a rank two parabolic, then the condition that  $\beta, \beta'$ be standard again ensures that  $\Pi_H(\beta), \Pi_H(\beta')$ are within bounded distance of 
$\eta$
and the result follows.
 \end{proof}

    \begin{prop} \label{quasiwordpath} 
Let $\G$ be a fixed geometrically finite Kleinian group with Cayley graph  ${\mathcal G}\Gamma$ relative to a  fixed set of generators  $\G^* = \{e_1, \ldots, e_k \}$, and let $\xi \in \Lambda_{\G}$.   Then there exists a sequence 
  $e_{i_1},e_{i_2} \ldots  \in \G^*$  such that 
  the path $\hat j_{\G}( e_{i_1} e_{i_2} \ldots)$ is a standard quasi-geodesic  in ${\mathcal G}\Gamma$ and such that $g_n\cdot  O  \to \xi$ in $\mathbb H^3 \cup \partial \mathbb H^3$.  \end{prop} 
\begin{proof} Let $\alpha \subset \HHH^3$ be the infinite hyperbolic geodesic from $O$ to $\xi$ and as usual let $\mathcal H(\alpha)$ be the set of thin parts traversed by $\alpha$. Exactly as in the proof of Proposition~\ref{floydorig} we can construct a path in $\Gr \G$ whose image under $j_{\G}$ tracks $\alpha \cup \mathcal H(\alpha)$
and which is ambient quasi-geodesic in the thick part $\mathcal V_{\G}$ of $\mathcal N_{\G}$. We can clearly ensure that in addition, segments of the path which track the boundary of rank two horoballs are standard. The result follows.
\end{proof}

\subsection{Proof of Theorem~\ref{thm:alg=ptwise}}
Suppose now that we are in the situation of   Theorem~\ref{thm:alg=ptwise}, that is, $\G$ is a fixed geometrically finite Kleinian group and  $\rho_n \co \G \rightarrow \PSL$ is a sequence
of discrete faithful weakly type preserving representations converging algebraically to $\rho_\infty \co \G \rightarrow \PSL $ and geometrically to $H$, and such that $H$ is geometrically finite.  (As noted in the introduction, this implies  $G_{\infty}$ is also geometrically finite.) Let $G_n = \rho_n (\G), n = 1,2,\ldots, {\infty}$. In this section we prove $EP(\xi)$ and $EPP(\xi)$, from which  Theorem~\ref{thm:alg=ptwise} follows.

The example to keep in mind is  that of a loxodromic $\rho_n(p), p \in \G $ converging to a parabolic $\rho_{\infty}(p) $, in such a way that 
suitable powers of the loxodromic also converge to another parabolic $q = q(p) \in H \setminus G_{\infty}$ which together with $\rho_{\infty}(p)$ generates a $\mathbb Z^2$-subgroup. This process is explained in detail in~\cite{marden-book} \S 4.9 and is what drives the well known Kerckhoff-Thurston examples~\cite{kerckhoff-thurston} of groups whose algebraic and geometric limits differ.  The main point is, that as $n \to \infty$  the translation length and rotation angle of  $\rho_n(p) $  are  related in such a way that 
for suitable powers $\rho_n(p^{m_n}), m_n \to \infty $,  the sequence $\rho_n(p^{m_n})$ converges to $q$.
In particular, $d_{\HHH}(O, \rho_n(p^{m_n}) \cdot O)$ is  bounded  while the word length $|p^{m_n}| \to \infty$. This of course violates  UEP. When $H$ is geometrically finite, this is the worst that can happen:
\begin{theorem} [\cite{jor-mar, marden-book} Theorem 4.6.1]
\label{extraelts} Suppose that $\G$ is   geometrically finite  and that  $\rho_n : \G \rightarrow \PSL$ are discrete, geometrically finite,  faithful representations converging algebraically to $\rho_\infty\co G \rightarrow \PSL $ and geometrically to $H$. If $H$ is also geometrically finite,  then 
it is generated by $\{ \G^*, q_1, \ldots, q_s\}$ where $\G^*$  is a generating set for $\G$ and $q_i$ are the `extra' parabolics  which arise in the limit as a result of the phenomenon above.
\end{theorem}

In the proof of Theorem~\ref{thm:strong=unif}, we used the crucial fact (Lemma~\ref{floyd1a}) that the image in $\HHH^3$ of a geodesic segment in $\Gr \G$ is uniformly an ambient quasi-geodesic in $\mathcal V_n = \mathcal V_{\epsilon,G_n} $. However in the above situation,  the elements $\rho_n(p^{m_n})$ in the 
approximating groups $G_n$ 
give paths 
which  track  a Margulis tube for distance $O(m_n)$, but whose initial and final points are  distance  $O(1)$ apart. Clearly such paths cannot be  ambient quasi-geodesics 
 with uniform constants independent of $n$. To get around this, we first approximate by 
taking a new set of generators $G^*_n$ for $G_n$  constructed so as to be close to those of the geometric limit $H$, and then substitute back in for the original generators of $\G$. Before doing this, however, we pause to discuss parabolic elements in $\G, G_n$ and  $G_{\infty}$.

\subsection{Parabolic blocks} \label{parabblock}The group $G_{\infty}$ contains a finite number of conjugacy classes of parabolic subgroups. Choose  $\mathcal P = \{p_1, \ldots, p_t \} \subset \G$ so that 
$\{\rho_{\infty}(p)\co   p \in \mathcal P\}$ contains one representative of each class, such that   the   horoball  $T^p_{\infty}$ based at the fixed point of 
$ \rho_{\infty}(p)$
intersects the Dirichlet domain $\mathcal D_{\infty}$, and such that $\rho_{\infty}(p)$ generates $Stab_{G_{\infty} } \,(T^p_{\infty})$. 
(We changed notation here from $H$ for a horoball  to $T$  for a tube, because a thin component which is the quotient of a  horoball in some $\HHH^3/G_n$ may be a tube in another. From now on we shall use the notation $T$ for both horoballs and equidistant tubes, the latter being the  lift of a Margulis tube to $\HHH^3$, see   immediately before the proof of Lemma~\ref{expdist1} for a precise definition.)

By the algebraic convergence, there is a group $G_{n_0}$ such that if $p \in \mathcal P$ then
$\rho_{n}(p)$ is  short in $G_n$ for all $ n \geq n_0$, in the sense that its length is less than the   constant $\e < \e_{\mathcal M}$ chosen above. Denote by 
$T^p_n$ the 
lift to $\HHH^3$ of the thin part $M_{thin}(\e)$ of $M= \HHH^3/G_n$  whose stabiliser is generated by $\rho_{n}(p)$.  Note that while by construction $T^p_{\infty}$ is a horoball,  $T^p_n$ may be either a horoball if $\rho_{n}(p)$ is parabolic or an equidistant tube from the short loxodromic
$\rho_{n}(p)$ otherwise. (We do not exclude the possibility that $\HHH^3/G_n$ may contain other thin parts or even horoballs  not associated to elements of $\mathcal P$.) By slightly reducing the choice of $\epsilon$  if necessary, we can choose $a>0$ such that the distance between any two thin parts  of $\mathcal N_n$ is at least $a$  for any $n \in \mathbb N \cup \infty$.

The set $\mathcal P$ is divided into two subsets, the set $\mathcal P_0 = \{p_1, \ldots, p_s \} $ which commute with the `extra' parabolics $ q_1, \ldots, q_s$
of Theorem~\ref{extraelts}, and the remaining set $\mathcal P_1 = \mathcal P \setminus \mathcal P_0$.
 By assumption the maps $\rho_{n}, \rho_{\infty}$ are weakly type preserving, in other words every   parabolic element $g \in  \G$  is also parabolic in $G_n, n = 1,2 \ldots, \infty $.  Note also that  a parabolic element in a Kleinian group   lies in a $\mathbb Z^2$-subgroup  if and only if it stabilises a rank two cusp. Since  the groups $\G, G_n, G_\infty$ are all abstractly isomorphic, this means that  rank two parabolics $\rho_\infty(p), p \in \mathcal P$ are `persistent', so that 
necessarily $p \in \G$ and $ \rho_n(p) \in G_n$ are also parabolic and moreover $p \in \mathcal P_1$.

We can choose $n_0 \in \mathbb N$ such that if $p \in \mathcal P$ then
$\rho_{n}(p)$ is  short  in $G_n$ for all $ n \geq n_0$, in the sense that its length is less than the Margulis constant $\e_{\mathcal M}$.  
We would like to renumber so that $G_{n_0} $ becomes $  \Gamma$.  However this causes a minor technical difficulty in the case in which $G_{n_0} $ contains parabolic elements which are not parabolic in either $\G$ or $G_{\infty}$, so that  we cannot replace the maps 
$\hat i_n \co \Lambda_{\G} \to \Lambda_{G_n}$ by a map $\Lambda_{G_{n_0}} \to \Lambda_{G_n}$. To simplify notation we renumber so that $G_{n_0} = G_0 , j_{n_0} = j_0$ etc, and live with the minor annoyance of two distinct groups $\G$ and $G_0$.

To make the next definition, we use the following generalisation of  Lemma~\ref{lemmaV}, whose proof is identical to the earlier version.  We write  $X=_{\G} Y$
to indicate that   $X$ is equal to $Y$ as elements in   $\G$, but not necessarily as words in $ \G^*$.  

\begin{lemma}
\label{lemmaVa} Let $\G, G_{0}$ be as above. There exists $c >0$ with the following property. 
Let $\lambda $ be a quasi-geodesic segment in $\mathcal{G} \Gamma$,  and let  $\pi$ denote projection from $j_{{0}}(\lambda)$ to $[j_{{0}}(\lambda)]$. Suppose that for some $ g, h \in \G$, the segment of $j_{0}(\lambda)$ from $j_{{0}}(g)$ to $j_{{0}}(h)$  projects to an arc contained in a single component $T \in \HH_{\epsilon,G _{0}}$. Then $g^{-1}h =_{\G} g^{-1}zgy$ where $z \in Stab_{G_0} \, T$ and $d_{\G}(1, y) \leq c$.
\end{lemma}

  \begin{definition} A geodesic segment $ e_{r+1} \ldots e_s $ in $\mathcal {G} \Gamma$
is called a \emph{parabolic block relative to $  p \in \mathcal P$} if  $ e_{r+1} \ldots e_s =_{\G} p^ky$ for some $y$ with $d_{\G}(1, y) \leq c$. We call $|k|$ the  \emph{length} of the block.  \end{definition}

 \begin{remark} {\rm   It is worth clarifying exactly  how this definition  relates to  the lemma. 
 Setting $g_r = e_{i_1} e_{i_2} \ldots e_{i_r}$,  the segment  $\hat j_{0}( e_{r+1} \ldots e_s) $ runs from $j_{0}(g_r)$ to $j_{0}(g_s)$. Set $g=g_r$ and $h= g_s$ in the lemma. Then  $ e_{r+1} \ldots e_s  = g_r^{-1} g_s = g_r^{-1}zg_r y$ where 
   $\rho_0( z) \in \Stab T$. Thus  $\rho_0( g_r^{-1}z g_r) \in \Stab \rho_0( g_r)^{-1}T$ and $\rho_0( g_r)^{-1}T = T_0^p$ for some $ p \in \mathcal P$. }\end{remark}


\begin{lemma}
\label{penetrates}  Given $D>0$, there exists $k_0 = k_0(D)>0$ such that  given  $k_1>k_0$, there exists $M = M(k_1,D) \in \mathbb N$  such that if
  $  p^ky, p \in \mathcal P_0$ is a parabolic block of length  $ k \in [k_0, k_1]$, then
any geodesic segment whose endpoints are at distance at most $D$ from 
$j_n(1)$ and $j_n(p^ky)$   respectively penetrates the thin part $T^p_n$ for all $n > M$.
If $ p \in \mathcal P_1$ then the same result is true for all  $ k \geq  k_0$.
 \end{lemma}
\begin{proof} 
 Since $y$ is in a bounded neighbourhood  of $1 \in {\mathcal G} \G$, using the algebraic convergence we get  a uniform bound  $d_{\HHH}(O, j_n(y)) \leq D_1$
for all $n$.   Hence \begin{eqnarray*}
  & &d_{\HHH}( j_n( p^ky), T^p_n) \leq d_{\HHH}( j_n( p^k ), T^p_n) + d_{\HHH}( j_n( p^ky),  j_n( p^k) \\ & &=    d_{\HHH}( j_n( p^k), T^p_n)+ d_{\HHH}( j_n( y),    O) \\
  & &=    d_{\HHH}( O, \rho_n( p^{-k}) T^p_n)+ d_{\HHH}( j_n( y),    O) =  d_{\HHH}( O,  T^p_n)+ d_{\HHH}( j_n( y),    O) .
  \end{eqnarray*}
Since  by Proposition~\ref{prop:diameters},  $O$ is a uniformly bounded distance  from $T^p_n$, this gives a uniform bound  $  d_{\HHH}( j_n( p^ky), T^p_n) \leq D_2$ say.
  
 Now let $A, B$ be points at distance at most $D$ to  $O = j_n(1)$ and $  j_n( p^ky)$ respectively,   and let $X, Y$ denote their projections
 onto $T^p_n$. Then $A,B$ are at distance at most $D+D_2$ to  $T^p_n$. Thus there exists a uniform constant $c>0$ such that if the distance between $X$ and $Y$ along $\partial T^p_n$  is at least $c$, then  $[A,B]$ penetrates $T^p_n$.

 If $ \rho_n(p)$ is parabolic, since $Y = \rho_n(p)^k \cdot X$ and since $\rho_n(p)$ translates a definite distance along $\dd T$  (because the injectivity radius of $\HHH^3/G_n$ at points on $\partial T/G_n$ is some fixed $\epsilon>0$), the result is now straightforward. However if $ \rho_n(p)$ is loxodromic, we have to take care that the multiplier is not such that $  d_{\HHH}( X, \rho^k_n(p) \cdot X)  = O(1)$.
 
 Let $G$ be a Kleinian group and let $S\in G$ be a loxodromic transformation with  multiplier bounded in modulus by some fixed $\theta_0 $. Let $T \subset \HHH^3$ be the equidistant tube corresponding to $S$, such that the injectivity radius of  $\HHH^3/G$ at points on the boundary of the image of $T$ in $\HHH^3/G$ is  $\epsilon>0$. Then  there exists a   constant $  c' >0$ such that provided $r\theta_0 < \pi/4 $,  the distance  between $Q  \in \partial T$  and $S^r(Q) \in \dd T$ along $ \partial  T$ is 
at least $c' r\epsilon$.   In particular, since the $T^p_n$ are all components of $\HH_{\epsilon;G_n}$ for  fixed $\epsilon$, this discussion applies to our present situation with $S=\rho_n(p)$.

 Choose $k_0 \in \mathbb N$ so that $c' k\epsilon > c$ whenever $k \geq k_0$. 
Since $\rho_n(p) \to \rho_{\infty}(p)$, the multiplier $\lambda_n$ of $\rho_n(p)$ converges to $1$.
Thus by the above discussion,  given $ k_1>k_0$   we can find $M = M( k_1)$, so that for all $n \geq M$, the argument of the multiplier of 
$\rho_n(p)$ is bounded in modulus by $\pi/4 k_1$ and hence so that the argument of the multiplier of
$\rho_n(p^r)$ is bounded in modulus by $\pi/4$ for all $|r| \leq  k_1$.

It follows that if
$X \in \partial T^p_n$ then,  provided $k_0 \leq k \leq  k_1$,  the distance from $X$ to $\rho_n(p^k) \cdot X$ along $ \partial  T^p_n$ is 
at least $c$ 
and so if $A,B$ are points at distance at most $D$ to  $O = j_n(1)$ and $  j_n( p^ky)$ respectively,   then the geodesic segment $[A,B]$   penetrates $T^p_n$.  

 Note that if $ p \in \mathcal P_0$ it is crucial here to have the fixed upper bound $k \leq  k_1$; in fact $d(X, \rho_n(p^k) \cdot X) \asymp 1$ for values of $k$ such that $\rho_n(p^k)$ is close to $q \in H$. If  $ p \in \mathcal P_1$ this restriction is unnecessary.
 \end{proof}

\subsection{Correspondence of generators}  We  now make a precise correspondence between a set of generators of $G_n$ and those of the geometric limit $H$.

\begin{lemma}\label{unique} There exists $\delta_0 >0$ such that for  any $h \in H$,  there  exists $m_0 =m_0(h)$ such that
there is a unique $g \in \G$  with $j_n(g)\in B_{\HHH}(h \cdot  O, \delta_0)$ whenever $ n \geq m_0$.
\end{lemma}
  \begin{proof}       Let  $\mathcal D_n,\mathcal D_H $ be the Dirichlet domains  for $G_n,H $ centre $O$ respectively. As in the proof of Proposition~\ref{prop:diameters}, 
the thick part of $ \mathcal  N_H \cap \mathcal D_H$ is contained in a ball $B_r$ of finite diameter in $\HHH^3$.  
 The groups $G_n$ converge geometrically, and hence polyhedrally, to $H$ (meaning that the faces of $\mathcal D_n$ converge to those of $\mathcal D_H $ uniformly on compact subsets of $\HHH^3$).  By the universal ball property, see~\cite{marden-book} Lemma 4.3.11, there exists $\delta>0$ such that  $\mathcal D_n \cap B_r$ contains the ball  $B_{\HHH}( O, \delta)$ for all $n$.
 Suppose that $g, g' \in \G$ are such that $j_n(g) ,j_n(g')   \in B_{\HHH}(h \cdot O, \delta/2)$ for some $h \in H$. Then $d(O, j_n(g^{-1} g')) < \delta$ from which it follows  that   
 $g= g'$.  That  $B_{\HHH}(h \cdot O, \delta/2) \cap G_n \cdot O \neq \emptyset$ for sufficiently large $n$ follows from the geometric convergence, proving the lemma with $\delta_0 = \delta/2$.
  \end{proof} 
  
Given any finite set $A \subset H$, Lemma~\ref{unique}  allows us to make a bijective correspondence between   $A$ and a corresponding subset    $ A_n \subset G_n$ for all sufficiently large $n$.  Choose a set of generators $\G^*$ for $\Gamma$ and set  $H^* = \{ \G^*, q_1, \ldots, q_s\}$ with $q_i$ chosen as in Theorem~\ref{extraelts}.
We define $G^*_n$ to be the corresponding set of elements in $G_n$. This is well defined for any $n > \max \{m_0(g): g \in \G^*\}$, with $m_0(g)$ as in Lemma~\ref{unique}. According to that lemma, each element of $G_n^*$ is either already in $\G^*$, or is close to $q_i$ for one of the additional generators as in Theorem~\ref{extraelts}. Since each such $q_i$ is the limit of a sequence of the form $\rho_n(p^{m_n})$ where $ p \in \mathcal P_0$ and $m_n \to \infty$, we may suppose that $n_0$ in Section~\ref{parabblock} above is also chosen so that the additional elements in $G_n^*$
are all of this form. Once again, we renumber so that $G_{n_0}$ becomes $G_0$.

\subsection{Proof of Theorem~\ref{thm:alg=ptwise}} Now we turn to the main part of the proof of Theorem~\ref{thm:alg=ptwise}.

For $g \in \G$,  we use $|g |_{\G}$ or $|g | $ to denote word length relative to the generating set $\G^* $  and $|g |_{n}$ to denote word length relative to    the generating set $G^*_n$.  
We shall need to distinguish between equality as group elements and equality as words (meaning that all letters are identical). We write $W=AB$ to mean that all letters of $W$ are exactly the same as those in the juxtaposed strings of letters $A,B$ and $W=_{G} AB$
to mean that   $W$ is equal to $AB$ as elements in a group $G$.

\begin{definition} Let $G$ be a Kleinian group with generators $G^*$ and let $|\cdot|_G$ denote the word metric.
We say a path $e_{i_1}e_{i_2} \ldots e_{i_s} ,  e_{i_j} \in G^*$ is $L$-\emph{quasi-geodesic} with respect to $(\mathcal {G} G, |\cdot|_G)$
if  $(b-a)/L \leq |e_{i_a} \ldots e_{i_{b-1}}|_G \leq L(b-a) $ for any $ 1 \leq a < b \leq s$.
\end{definition}

\begin{lemma}
\label{compare} Suppose that $W = w_{i_1}w_{i_2}\ldots w_{i_s}, w_{i_j} \in \G^*$ is $L$-quasi-geodesic in $(\mathcal {G} G_n, |.|_{n})$ for some $n \geq 0$.Then $W$ is also $L$-quasi-geodesic  in $(\mathcal {G} \Gamma, |.|_{\G})$.
 \end{lemma}
\begin{proof} Let $V =  w_{i_a} \ldots w_{i_{b-1}}$ be a subsegment of $W$.
Then
$b-a  < L |V|_n \leq L |V|_{\G}$ where the last inequality follows since any expression in $\G^*$ is also one in $G^*_n$. Since $w_{i_j} \in \G^*$ for all $j$, the inequality $  |V|_{\G} \leq b-a$ is obvious.
\end{proof}

 The key step in the proof of Theorem~\ref{thm:alg=ptwise} is the following proposition, which says that although the image  $j_n(\lambda)$ of a quasi-geodesic $\lambda \subset \Gr \G$ may not be an ambient quasi-geodesic in $\mathcal V_n$ (see the discussion following the statement of Theorem~\ref{extraelts}), it is still within uniformly bounded distance of $ [j_n(\lambda)] \cup \mathcal H([  j_n(\lambda)])$. The idea is first to track $[j_n(\lambda)]$ with uniform bounds by   word paths $W_n$ in the generators $G^*_n$, and then to study carefully how these word paths look  when rewritten in terms of the generators $\G^*$. This second step is the content of Proposition~\ref{PropA}, which is needed in the proof of Proposition \ref{Prop1}. 
  
\begin{proposition} \label{Prop1} Suppose that $\lambda$ is a standard  $K$-quasi-geodesic in $\mathcal {G} {\Gamma}$.  Then there exists $n_1 \in \mathbb N$ such that $  j_n(\lambda)$ is at uniformly bounded distance to $ [j_n(\lambda)] \cup \mathcal H([  j_n(\lambda)])$ for all $n \geq n_1$, with a bound $D$ that depends on $K$ but  is independent of $n$.
 \end{proposition}

  \begin{proof} To simplify, we may as well assume that the initial point  of  $ \lambda$ is at  $1 \in \Gr \G$. Denote the final point by $g_{\lambda} \in \Gamma$.
 Floyd's construction in Proposition~\ref{floydorig}  allows us to find a standard word path $W$ in the generators $G_n^*$  so that $W =_{\Gamma}g_{\lambda}$,  such that 
$ j_n(\lambda)$ is  at bounded distance $K_n$ say to $[  j_n(\lambda)] \cup \mathcal H[(  j_n(\lambda))]$
  and whose length is comparable to $| \lambda|_n$, the length of $\lambda$ relative to the generators $G_n^*$.     
   The constants involved in determining $K_n$ are:
  the diameter $d_n$ of the thick part of the truncated Dirichlet  domain $\mathcal D_n \cap \mathcal V_n$, $a_n = \max \{d(O, g \cdot O) \co  g \in G^*_n\}$ and $C_n = \max \{|g|_n \co d(O, g \cdot O) \leq D\}$ for some uniform $D>0$. (It is crucial here that word length $|\cdot |_{n}$ is defined using the generating set $G^*_n$.) Then $d_n$ is uniformly bounded by Proposition~\ref{prop:diameters}, 
  $a_n$ is uniformly bounded by the \emph{geometric} convergence since so is 
  $\max \{d(O, h \cdot O) \co g \in H^*\}$, see Lemma~\ref{unique},  and finally, again in virtue of Lemma~\ref{unique}, $C_n$  is uniformly bounded
  by   $C_H= \max \{|h|_H \co h \in H \co d(O, h \cdot O) \leq 1 + 2d_n+ \delta\}$, where $\delta$ is as in Lemma~\ref{unique}. Thus all constants involved are uniform in $n$ and hence
there exists $L>0$ such that   the path $W$ (which is written in terms of the generators $G^*_n$) is $L$-quasi-geodesic in $(\Gr G_n, |\cdot|_n)$ for all $n$.

The only letters in $W$ which are not also generators of $\G^*$ correspond to elements in $G^*_n \setminus \G^*$  and are therefore of the form 
$u= p^k$ for some $ p \in \mathcal P_0$. 
Rewriting $W$ in terms of  the generators $\Gamma^*$ by substituting these terms we find $W =_{\G}  V_{1}U_{1} V_2U_2\ldots V_s $ where  all of the letters in $V_i$ belong  to $\G^*$, and where $U_i = u_i^{a_i}= p^{a_im_i}$ for some $ p \in \mathcal P_0$.
 In this expression we may have $V_i = \emptyset$ but in this case (by combining adjacent terms belonging to the same parabolic) we assume that 
 $U_i, U_{i+1}$ are associated to distinct elements of $  \mathcal P_0$.  We define $W_{\G}$ to be the word $  V_{1}U_{1} V_2U_2\ldots V_s $ in $\G^*$ and emphasize once again that  
 $W$ and $ W_{\G}$ both depend  heavily on $n$.

Since $V_i$ is by construction $L$-quasi-geodesic in $(\Gr G_n, |\cdot|_n)$, by Lemma~\ref{compare} it is also $L$-quasi-geodesic in $(\Gr \G, |\cdot|_{\G})$.   We would like to claim that $W_{\G}$ is quasi-geodesic in $(\Gr \G, |\cdot|_{\G})$.
 However this may not be the case because, as illustrated in Figure~\ref{fig:backtracking}, there could be a segment  which is a power of $p_i $ at the  end of   $V_i$ which cancels into $ U_i$, and likewise a segment   at the beginning of $V_i$ which cancels into $ U_{i-1}$. To remedy this we will show in Proposition~\ref{PropA} below that it is possible to split $V_i$ into blocks
 $V_i= X_iY_iZ_i$ where  $  X_i $ cancels into $ U_{i-1}$ and $Y_i $ cancels into $ U_{i}$, in such a way that  $Z_iU_i X_{i+1} =_{\Gamma}  \hat U_i $ and $j_0(  Y_{1}\hat U_1 Y_2\hat U_2\ldots Y_s )$ is a standard ambient $L'$-quasi-geodesic in $\mathcal V_0$, for some $L'$ depending on $L$ but independent of $n$. The essential idea is illustrated in Figure~\ref{fig:backtracking}. Notice that  following through the definitions,    $g_{\lambda} =_{\G}  Y_{1}\hat U_1 Y_2\hat U_2\ldots Y_s  $.

Assuming this we proceed as follows.  By hypothesis,  $\lambda$ is  a standard quasi-geodesic in $\mathcal {G} {\Gamma}$ with initial point $O$ and endpoint $g_{\lambda} \cdot O$. Thus  by construction $\hat j_{0}(\lambda)$ has the same endpoints as $\hat j_{0}(Y_{1}\hat U_1 Y_2\hat U_2\ldots Y_s) $.  
Hence by Lemma~\ref{ambientdist}, the paths $\hat j_{0}(Y_{1}\hat U_1 Y_2\hat U_2\ldots Y_s) $, $\hat j_{0}(\lambda)$ are  a bounded distance apart.  Since $j_{n}j_{0}^{-1}\co  \Gr \G \to \HHH^3$ is uniformly Lipschitz because of the algebraic convergence, $\hat j_{n}(\lambda )$  is a uniformly bounded distance from $\hat j_{n}(Y_{1}\hat U_1 Y_2\hat U_2\ldots Y_s) $, with a bound that depends on $K$ and $L'$.

Now  we show that   any point $Q \in \hat j_{n}(  Y_{1}\hat U_1 Y_2\hat U_2\ldots Y_s) $
is at uniformly bounded distance from $[j_n(\lambda)] \cup \mathcal H ([j_n(\lambda)])$.
We have $ Q = j_n(A)$ where $A$ is a subsegment of $Y_{1}\hat U_1 Y_2\hat U_2\ldots Y_s$ starting at $1$.
  Those subwords starting at $1$ and ending in a letter in some $\hat U_i$ have images under $ \hat j_{n}$ within uniformly bounded distance of the corresponding thin part $T_i\in \mathcal H ([j_n(\lambda)])$, 
 while the image of  points corresponding to subsegments
starting  from $1$  and ending at a letter contained in a segment $Y_i$ are within uniformly bounded distance of $[j_n(\lambda)] \cup \mathcal H ([j_n(\lambda)])$ since any word of the form $Y_{1}\hat U_1 Y_2\hat U_2\ldots Y'_r$  with $Y'_r$ an initial subsegment of $ Y_r, r \leq s$, is by construction equal to a subword of $W$.

We have shown that $\hat j_{n}(\lambda )$  is a uniformly bounded distance from  $\hat j_{n}(Y_{1}\hat U_1 Y_2\hat U_2\ldots Y_s)$ and that any point  in $ \hat j_{n}(  Y_{1}\hat U_1 Y_2\hat U_2\ldots Y_s) $
is at uniformly bounded distance from $[j_n(\lambda)] \cup \mathcal H ([j_n(\lambda)])$.
The result follows.
   \end{proof}

To prove Proposition~\ref{PropA} we need two easy lemmas.

 \begin{lemma} \label{quasigeod1}   There exists $\ell_0>0$ with the following property. 
 Suppose that 
 $T_i, i=1,\ldots, s$ is a sequence of thin parts $T_i \in \mathcal {H}_{\epsilon, G_0}$ such that  $T_i \neq T_{i+1}$ for all $i$.
  Suppose that  $[Q_iP_{i+1}]$ is the common perpendicular to   $T_i,T_{i+1}$ and that $\mu$ is the piecewise geodesic arc joining points $Q_1,P_2,Q_2, \ldots ,Q_{s-1}, P_s$,  and suppose also that the geodesic segment $[P_iQ_i]$ has length at least $\ell_0, i=2, \ldots, s-1$. Then $\mu$ is a quasi-geodesic  in $\HHH^3$.\end{lemma}
  \begin{proof}  The angle between the segments $[P_iQ_i]$ and $[Q_iP_{i+1}]$ is at least $\pi/2$ and by the choice of $\epsilon$ (see the discussion in Section~\ref{parabblock}), $[Q_iP_{i+1}]$ has length at least $a>0$. The result is now standard.
  \end{proof}
  
  \begin{lemma} \label{quasigeod2}  Suppose that $D>0$ is given and that $T,T' \in  \mathcal {H}_{\epsilon, G_0}$. Let $P,P'$ be points within distance $D$ of $T,T'$ respectively such that the geodesic segment  $ [P,P']$ is disjoint from both $T$ and $T'$. Then $ [P,P']$ is within bounded distance of the common perpendicular to  $T$ and $T'$ with a bound that depends only on $D$ and $\epsilon$.\end{lemma}
  \begin{proof}  Let $X,X'$ be the nearest points to $P,P'$ on $T,T'$ respectively and let $Q \in T,Q' \in T'$ be the endpoints of the  common perpendicular to  $T$ and $T'$. Consider the quadrilateral with vertices $X,Q,Q',X'$. Since the angles at $Q,Q'$ are at least $\pi/2$ and since $d_{\HHH}(Q,Q') >a$, the geodesic segment $[XX']$ is at bounded distance to the union of the arcs $[XQ], [QQ'], [X'Q']$. Since
 $[XQ]$ penetrates $T$,  so does $[PP']$ unless  $[XQ]$ has bounded length; likewise with $[X'Q']$. 
  \end{proof}

  \begin{proposition} \label{PropA} Given $L>0$, there  exist  $L' = L'(L)>0$ and $n_1 \in \mathbb N$ with the following property. 
  Suppose given a word $   V_{1}U_{1} V_2U_2\ldots V_s $  in $\G^*$  as in the proof of Proposition~\ref{Prop1}, that is, so that  $V_i$ is either empty or is $L$-quasi-geodesic in $(\Gr \G, |\cdot|_{\G})$, and where $U_i =   p_{j_i}^{a_im_{j_i}}$ for some $ p_{j_i} \in \mathcal P_0, a_i \in \mathbb Z, a_i \neq 0$, and so that  if $V_i = \emptyset$ then 
 $ p_{j_i} \neq  p_{j_{i+1}}$ for all $i$. 
  Then we can split $V_i$ into blocks as $V_i= X_iY_iZ_i$  (where any of $X_i, Y_i$ or $ Z_i$ may be empty) in such a way  that  there exists a word $\hat U_i =_{\Gamma}  Z_iU_i X_{i+1} $ in the generators $\G^*$ which is a parabolic block relative  to $ { p_{j_i} }$  and  such that $  \hat j_0(Y_{1}\hat U_1 Y_2\hat U_2\ldots Y_s) $ is a standard ambient $L'$-quasi-geodesic in $\mathcal V_0$, for all $n \geq n_1$.  
\end{proposition}
 \begin{proof}  
We continue with the notation and discussion of Proposition~\ref{Prop1}. The statement and the idea of the proof are illustrated in Figure~\ref{fig:backtracking}.
  Throughout the proof, when we say that  various distances are `bounded', we mean they are uniformly bounded in $n$ and independent of the choices of $\lambda$ and words $W$.   
  
  Let $g_i = \rho_{0}(V_{1}U_{1} V_2U_2\ldots U_{i-1}  )$ and $g'_i = \rho_{0}(V_{1}U_{1} V_2U_2\ldots V_{i} )$.
 For $ i = 1, \ldots, s$,  let  $ \beta_i  = [j_{0}(g_i),   j_{0}(g_i' )]$ and $  \gamma_i = [ j_{0}(g_i' ), j_{0}(g_{i+1} )]. $
Thus  $\beta_i =   g_i ( [j_{0}(V_i)])$  and 
$\gamma_i =  g'_i( [j_{0}(U_i)])$.  Recall that each word $U_i$ corresponds to a thin part $T_i \in \HH_{\epsilon, G_n}$.

Temporarily let us drop the subscript $i$ and    write $V = V_i$ etc. For simplicity we will work with $\beta= [j_{0}(V)] = g_i^{-1}(\beta_i)$ rather than $\beta_i$.  Let $A,  D$ be  the initial  and final points of $\beta$.

\begin{figure}[hbt] 
 \centering
\includegraphics[height=6cm]{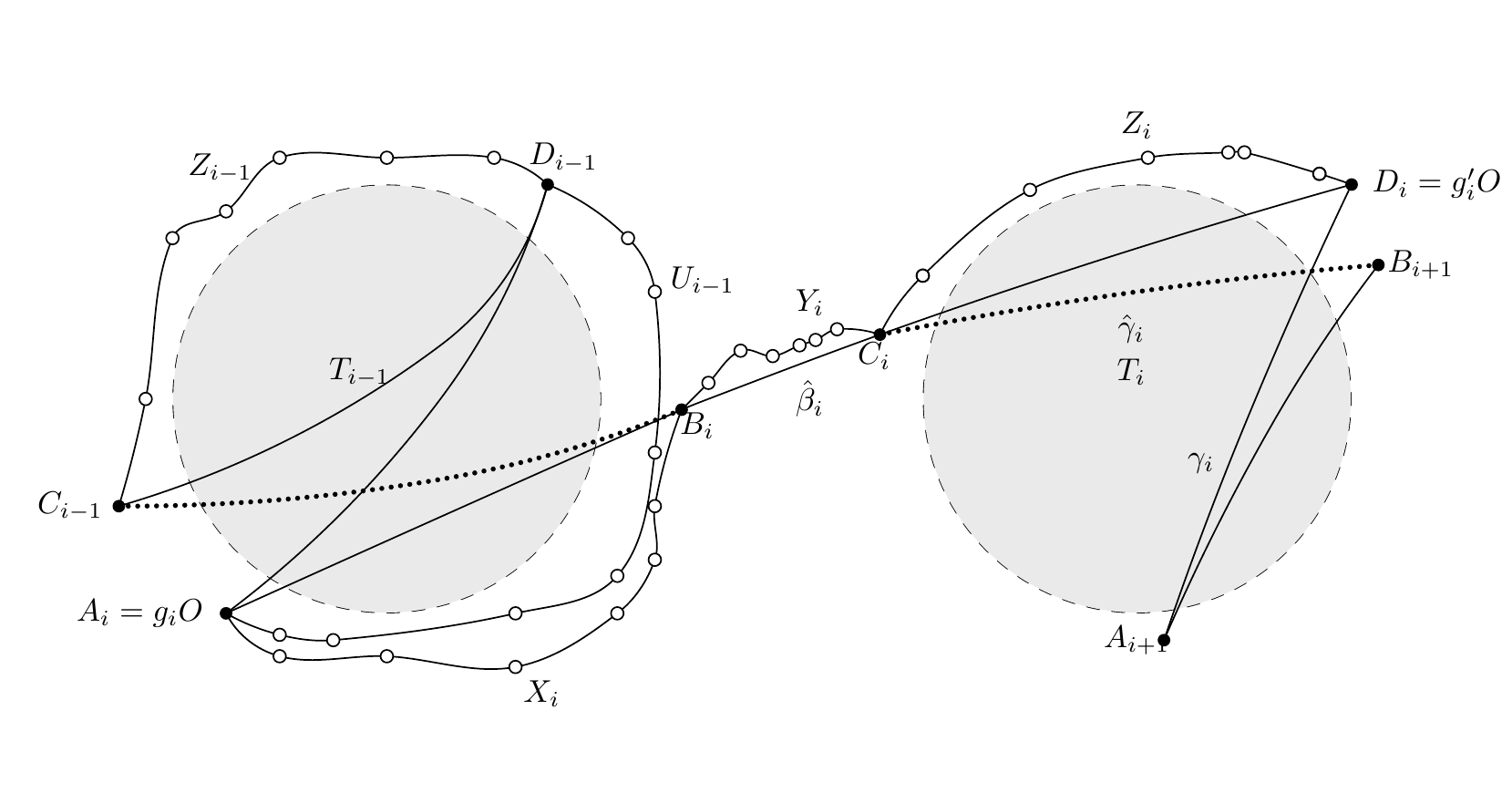} 
\caption{A segment of the path  $j_{0}(W)$  showing the segment $V_i$ from $A_i$ to $D_i$ split as  $V_i = X_iY_iZ_i$. Note the cancellation  in the path $U_{i-1}V_i$ from $D_{i-1}$ to $A_i$ to $D_i$. We shorten $\beta_i = [A_i,D_i]$ to $\hat \beta_i = [B_i,C_i]$ and replace $[D_{i-1},A_i] = \gamma_i$ by $[C_{i-1},B_i] = \hat \gamma_i$. The key point is to see  $\ell(\hat \gamma_i) \to \infty$ independent of the choices made.}
\label{fig:backtracking}
 \end{figure}

 Let $T^1,T^2$ denote respectively the (distinct) horoballs $g_i^{-1}(T_{i-1}), g_i^{-1}(T_{i})$. It follows from the construction  that $A$ is a bounded distance from $T^1$ and $D$ is a bounded distance from $T^2$.
Consider the projection $\pi$ from  $\hat j_{\G}(V)$ to  $\beta$. Set $\beta^j = \beta \cap T^j, j=1,2$.  
Now define orbit points $B,C$ on $\hat j_{\G}(V)$ as follows: if 
$\beta^1 = \emptyset $ then $B=A$, otherwise $B$ is a nearest orbit point to the final point of  $\pi^{-1}(\beta^1)$; likewise  if 
$\beta^2 = \emptyset $ then $C=D$,  otherwise
$C$ is a nearest orbit point to the first point on $\pi^{-1}(\beta^2)$.
   (We remark that it is possible that $C$ precedes $B$ in order along $\hat j_{\G}(V)$. In this case by Lemma~\ref{penetrates} there is a bounded distance between $B$ and $C$.
Thus up to changing constants by a bounded amount, we can replace $C$ by the point $B$.)
  
Denote the segments of $V$ from $A$ to $B$, from $B$ to $C$, and from $C$ to $D$ by $X, Y, Z$ respectively. Thus $V = XYZ$.
By Lemma~\ref{lemmaVa}, since $A$ and $B$ are at a bounded distance from $T^1$, the segment $X$ is a parabolic block relative to  $T^1$, and similarly for $Z$ relative to $T^2$.

Now let $A_i, B_i, C_i, D_i, X_i, Y_i, Z_i$ be the images of $A,B,C,D$ and   $ j_{0}(X)$, $ j_{0}(Y)$ and $ j_{0}(Z)$ under $g_i$. 
  Let $\hat \gamma_i$ denote the geodesic from $C_{i-1}$ to $B_i$. We claim that (independent of all the many choices made) 
its hyperbolic length  $\ell(\hat \gamma_i) \to \infty$ uniformly with $n$. In particular, given $\ell_0 >0$ we can choose $n_1 $ so that  $\ell(\hat \gamma_i) \geq \ell_0$, whenever $n \geq n_1$.
To prove this, it is clearly enough to show that $ |C_{i-1}B_i|_{\G} \to \infty$ as $n \to \infty$. We have 
\begin{eqnarray*} &&  |C_{i-1}B_i|_{\G} \geq   |C_{i-1}B_i|_{n} \succ \\ &&  |C_{i-1}D_{i-1}|_{n} + |D_{i-1}A_i|_{n} +  |A_{i}B_{i}|_{n}=  (|C_{i-1}D_{i-1}|_{\G} + | A_iB_i|_{\G}) +  |U_{i-1}|_{n},\end{eqnarray*} where 
the second inequality is because the path $C_{i-1}D_{i-1}A_iB_i$ is by definition $L$-quasi-geodesic in $G_n^*$ and the final equality follows since by definition the words $Z_{i-1}$ corresponding to the path $C_{i-1}D_{i-1}$ and $X_i$ corresponding to the path $D_{i-1}A_i$ are quasi-geodesic words in $G_n^*$ which happen to be expressed entirely by generators in $\G^*$.
Since
$$|C_{i-1}D_{i-1}|_{\G} +  |A_{i}B_{i}|_{\G} \geq   |D_{i-1}A_i|_{\G}-  |C_{i-1}B_i|_{\G} $$
we have
 $$2  |C_{i-1}B_i|_{\G} \succ    |D_{i-1}A_i|_{\G} +    |U_{i-1}|_{n} = m_{j_{i-1}} |U_{i-1}|_n +    |U_{i-1}|_{n}, $$  where we  used $ |D_{i-1}A_i|_{\G} = |U_{i-1}|_{\G} = m_{j_{i-1}}|U_{i-1}|_{n}$, as in the proof of Proposition~\ref{Prop1}.   Now by definition $ m_{j_{i}}$ is the exponent such that $p_{j_i}^{m_{j_{i}}} \to q_{j_i}$ in the geometric limit, so  $m_{j_{i}}\to \infty$ with $n$. Hence  $\ell(\hat \gamma_i) \to \infty$ with $n$ as claimed.

Now let $\hat \beta_i$ denote the   geodesic from $B_{i}$ to $C_i$. Note that $\hat \beta_i$ has endpoints within bounded distance of the segment of $\beta_i$ which is outside both $T_{i-1}$ and $T_i$. We claim that for sufficiently large $\ell_0$,  the path $\beta$ obtained by concatenating $\hat   \beta_1, \hat\gamma_1, \hat\beta_2, \ldots, \hat\gamma_s$ is quasi-geodesic in $\HHH^3$,   whenever $n \geq n_1$.
By construction, the endpoints of  $\hat\beta_i$ are within bounded distance of $T_{i-1}$ and $T_i$ respectively. Moreover by construction the segment $\hat\beta_i$  is outside both $T_{i-1}$ and $T_i$. Thus we are in the situation of Lemma~\ref{quasigeod2}
so that $\hat\beta_i$ is within bounded distance of the common perpendicular to 
$T_{i-1}$ and $T_i$. Adjusting the endpoints of each $\gamma_i$ by at most a uniformly bounded amount, we see we are in the situation of Lemma~\ref{quasigeod1}, and the result follows.

From the construction,  the path $\hat j_{0}(Y_i)$  tracks $\hat \beta_i$ at bounded distance.  
Now consider the segment $\hat j_0(Z_{i} U_i X_{i+1})$ from $C_i$ to $B_{i+1}$. We claim that this can be replaced by a path $\hat j_0( \hat U_i )$ with the same initial and final points, and where $\hat U_i =_{\G} Z_{i} U_i X_{i+1}$  is a  parabolic block relative to $p_{j_i}$. 
By construction the initial and final points $C_i$ to $B_{i+1}$ are at bounded distance to 
$T_i$. Hence by the method of Lemma~\ref{lemmaVa},  if $C_i = h_iO$ to $B_{i+1} = h'_{i+1}O$ then $(h'_{i+1})^{-1} h_i$ is a parabolic block relative to $p_{j_i}$, proving the claim.


 Now $j_0(Y_i)$ tracks $ \hat \beta_i $ and $j_0(\hat U_i)$ tracks the shortest path from $C_i$ to $B_{i+1}$ on $\dd T_i$ at bounded distance,   $\hat \gamma_i$ being the geodesic with the same endpoints. Since the concatenation $\beta$ of $ \hat \beta_1, \hat \gamma_1,\ldots, \hat \beta_s, \hat \gamma_s  $  is quasi-geodesic for $n \geq n_1$ 
 it follows that 
 $\hat j_{0}(Y_{1}\hat U_1 Y_2\hat U_2\ldots Y_s)$  is an ambient quasi-geodesic in $\HHH^3$. 
All constants involved are independent of the various choices made and of $n$. This completes the proof.
\end{proof}

We are finally ready to prove Theorem~\ref{thm:alg=ptwise}.  
Let $e_{i_1} e_{i_2} \ldots $ be a standard quasi-geodesic path in $\Gr \G$ such that $e_{i_1} e_{i_2} \ldots e_{i_k} \cdot O \to \xi $.   There are three possibilities:
\begin{enumerate}
\item the length of parabolic blocks in $e_{i_1} e_{i_2} \ldots $ is bounded above;
\item $e_{i_1} e_{i_2} \ldots $ contains parabolic blocks of arbitrarily long length;
\item $e_{i_1} e_{i_2} \ldots $ terminates in an infinite parabolic block.  
\end{enumerate}

It is straightforward to see that case (3) happens if and only if $\xi$ is a parabolic fixed point,  in which case $\hat i_n(\xi) \to  \hat i_\infty(\xi)$ follows immediately from the algebraic convergence. Thus we have only to  prove $EP(\xi)$ and $EPP(\xi)$ for cases (1) and (2). We start with
$EP(\xi)$. As usual, denote $\hat j_{\G}(e_{i_1} e_{i_2} \ldots )$ by $[1,\xi)$.

\begin{prop} \label{EPxi1} \label{quasigeod} Let $  e_{i_1} e_{i_2} \ldots $ be a standard quasi-geodesic $\Gr \G$ in which there is an upper bound on the length of parabolic blocks. Then
$((\rho_n),[1,\xi))$ satisfies
  $EP(\xi)$. 
 \end{prop} \begin{proof}
We have to show that given $N \in \mathbb N$, there exists $f_{\xi}(N) \to \infty$ as $N \to \infty$ such that $d(O, j_n(e_{i_1} e_{i_2} \ldots e_{i_N} )) \geq f_{\xi}(N)$ for any $n \in \mathbb N$.  If the result is false, there exist $A>0$ and $n_k, N_k \to \infty$ such that $d(O, j_{n_k}(e_{i_1} e_{i_2} \ldots e_{i_{N_k}} )) \leq A$. Passing to a subsequence, we may assume that $\rho_{n_k}(e_{i_1} e_{i_2} \ldots e_{i_{N_k}})$ converges geometrically to some $h \in H$. By Theorem~\ref{extraelts}, we must have $ h = h_{i_1} \ldots  
  h_{i_p}$ where $h_{i_j} \in H^*$, and since $N_k \to \infty$ we must have 
   $h_{i_j} \in  \{ q_1, \ldots, q_s\}$ for some $j$. It follows that $e_{i_1} e_{i_2} \ldots e_{i_{N_k}}$ contains arbitrarily long parabolic blocks contrary to hypothesis.
  \end{proof}

\begin{prop}  \label{EPxi2} Suppose that  $e_{i_1} e_{i_2} \ldots $ contains arbitrarily long parabolic blocks. Then
$((\rho_n),[1,\xi))$ satisfies
  $EP(\xi)$. 
  \end{prop}
  \begin{proof} 
 By Proposition~\ref{Prop1},  $j_n(e_{i_1} e_{i_2} \ldots e_{i_N})$ is at uniformly bounded distance $D$ say to $\alpha_n \cup \mathcal H(\alpha_n)$, where $\alpha_n = [j_n(e_{i_1} e_{i_2} \ldots e_{i_N})]$. Choose $k_0 = k_0(D)$ as in Lemma~\ref{penetrates}.
    Define $f_{\xi}(N)  $ to be the number of parabolic blocks of length at least $k_0$
in $e_{i_1}e_{i_2} \ldots e_{i_N}$. 
By our assumption, $f_{\xi}(N)\rightarrow\infty$ as $N\rightarrow\infty$.

Let $k_1 = k_1(N;\xi)$ be the maximum length of these $f_{\xi}(N)$ blocks.
By Lemma~\ref{penetrates}, there exists  $M = M(k_1,D) = M_{\xi}(N)$ so that $\alpha_n$ penetrates $\rho_n(g)T^p_n$
 provided $n \geq M$.  Recall there is a constant $a>0$ such that the distance between any two components  of $\mathcal \HH_{\epsilon, G_n} $ is at least $a$  for any $n$. It follows that  $d(O, j_n(e_{i_1}e_{i_2} \ldots e_{i_r})) \geq Na$ provided $ n \geq M_{\xi}(N)$ and the result follows.
\end{proof}

\begin{proposition} \label{Prop3}The pair
$((\rho_n),\xi)$ satisfies
  $EPP(\xi)$. 
  \end{proposition}
  \begin{proof}
   We have to show that  there exists $f_{1,\xi}(N) = f_1(N)$ such that $f_1(N)\rightarrow \infty$ as $N\rightarrow \infty$ and 
such that for any geodesic subsegment $\lambda$ of $[1,\xi ) $ lying outside $B_{\G}(1;N)$,
the $\Hyp^3$-geodesic  $[j_n(\lambda)]$  lies outside $B_{\HHH} (O; f_1(N))$ in $\Hyp^3$.

Use Propositions~\ref{EPxi1} and ~\ref{EPxi2} to find $f_1(N)$ such that  $d_\Gamma (1,g) \geq N$ implies  $d_{\Hyp} (j_n(g) , O) \geq
f_1(N)$ for all $ n \geq M_{\xi}(N)$.  Then $d_{\Hyp} (\rho_n(\lambda)  , O) \geq
f_1(N)$  for all $ n \geq M_{\xi}(N)$.

By Proposition~\ref{Prop1}, $j_n(\lambda)$ is at uniformly bounded distance $D$ to $[j_n(\lambda) ]\cup \mathcal H([j_n(\lambda) ])$.  Hence the entry and exit points of 
$[j_n(\lambda) ]$ to any thin component $T \in \mathcal H([j_n(\lambda) ])$ are outside $B_{\Hyp} (O ,  
f_1(N) -D)$, as is the sub-path of $j_n(\lambda)$  joining them. So by Lemma~\ref{tubes}, $[j_n(\lambda) ] $ is outside $B_{\Hyp} (O ,  
f_1(N)/4 -c)$ for a suitable uniform constant $c$. The proof follows as in Theorem~\ref{floyd}.
\end{proof}

  \begin{proof} [Proof of Theorem~\ref{thm:alg=ptwise}] 
  This follows from Proposition~\ref{Prop3}  and Theorem~\ref{ptwisecrit}.\end{proof}


\section*{Appendix: Hyperbolic Geometry Estimates} \label{sec:estimates}

\begin{appendthm}\label{easy} In the ball model with $O$ as centre, suppose that $X,Y \in \HHH^3$ lie outside $B_{\HHH}(O;R)$  and that the geodesic $[X,Y  ]$ joining them is also outside $B_{\HHH}(O;R)$. Then 
$d_{\mathbb E}(X,Y) \prec  e^{-R}$.
\end{appendthm}
\begin{proof} In the ball model $\mathbb B$, let $\xi,\eta$ denote the endpoints of the radial lines from $O$ through $X,Y$ on $\partial \HHH^3$. Let $P$ be the footpoint of the perpendicular from $O$ to the geodesic segment $[XY]$ and let $P'$ be the endpoint of this ray at $\infty$.
It will be sufficient to show that $d_{\mathbb E}(P,X) \prec e^{-R}$.

 Let $\theta = \angle XOP$. Then 
$$\tan \theta = \frac{\tanh d_{\mathbb H}(X,P) } {\sinh d_{\mathbb H}(O,P)} < e^{-R}/2,$$
 from which it follows that $d_{\mathbb E}(\xi,P') < e^{-R}/2$.

Now from $R  \leq d_{\mathbb H}(O,X)$ 
 we find easily $d_{\mathbb E}(X,
\xi) \prec e^{-R}$. The result follows.
\end{proof}

The following lemma  allows us to replace  orbit points by attracting fixed points.

 \begin{appendthm}\label{easy1}  Suppose that $A \in \PSL$ is loxodromic (resp. parabolic), and that  
$d_{\mathbb H}(O,A \cdot O) > R$, where $O$ is the centre of the ball model $\mathbb B$. Let $A^+ \in \partial \mathbb B$  be the attracting  fixed point (resp. fixed point) of $A$. Then $d_{\mathbb E}(A \cdot O,A^+) \prec  e^{-R/2}$.
\end{appendthm}
\begin{proof}  We use a fact  which we learned from ~\cite{marden-book} Lemma 1.5.4:
  in $\mathbb B$, the isometric circle $I_{A^{-1}}$ of $A^{-1}$ is the perpendicular bisector $L $ of the line from $O$ to $A \cdot O$. (To see this, let $\tau$ denote inversion in  $L $. Then 
$A^{-1} \tau$ fixes $O$ and hence is a Euclidean isometry. Using this together with the fact that  $|D\tau |_{L } = 1$, it follows easily from the chain rule that $|DA^{-1}|_{L } = 1$.)

Now $I_{A^{-1}}$ contains both $A  \cdot O$ and the attracting fixed point $A^+$. 
Since $d_{\mathbb H}(O,A \cdot O) > R$ we have  $d_{\mathbb H}(O,I_{A^{-1}}) > R/2$, hence the Euclidean diameter of $I_{A^{-1}}$ is $O(e^{-R/2})$ and the result follows.
\end{proof}

We need estimates of the distortion caused by skirting around horoballs or Margulis tubes.  If $L \subset \HHH^3$ is a geodesic we call  $T  = \{ x \in \HHH^3: d_{\HHH}(x, L) \leq R \}$  the \emph{equidistant tube of radius $R $} around $L$.    If $H$ is a horoball or  $T$ is a tube,  and if  $P_1,P_2 \in \partial H$ (resp. $ P_1,P_2 \in \partial T$), we denote by $ d_{\dd H}(P_1,P_2)$ (resp. $  d_{\dd T}(P_1,P_2)$)
  the length of the shortest path on $\dd H$ (resp. ${\dd T}$) from $P_1$ to $P_2$.

\begin{appendthm} \label{expdist1} Let $H \subset \HHH^3$ be a horoball and suppose $P_1,P_2 \in \partial H$. Set $l= d_{\dd H}(P_1,P_2)$ and   $d= d_{\HHH }(P_1,P_2)$. Supose that $d > d_0$ for some fixed $d_0 > 0$.
Then $l \asymp e^{d/2}$  with constants depending only on $d_0$. \end{appendthm}
\begin{proof} As is easily checked by explicit computation using the angle of parallelism formula, $l = 2 \sinh {d/2}$, see for example~\cite{floyd} p.\ 213.  
\end{proof}

One can make a  similar  estimate for tubes of sufficiently large radius.

\begin{appendthm}\label{expdist2}
Fix positive constants $  R_0,d_0$ and $h_0$. Let $T \subset \HHH^3$ be the equidistant tube of radius $R \geq  R_0$  around a  geodesic $ L \subset \HHH^3$ and let $P_1,P_2\in \partial T$.   Let $l= d_{\dd T}(P_1,P_2)$ and   $d= d_{\HHH }(P_1,P_2)$. Then  $l  \prec e^{d /2}$  where  the constant involved depends only on $R_0$. If moreover the distance between the projections of $P_1,P_2$ onto the axis $L$ of $T$ is at  most $h_0 $ and $d \geq d_0$, then $l  \asymp e^{d /2}$  where  the constant involved  depends only on $R_0,h_0$ and $d_0$.
\end{appendthm}
\begin{proof}   For convenience we arrange  things so that $L$ is the line from the origin   $\omega \in  \mathbb C$ to $\infty$ in the upper half space model of $\HHH^3$. 
 Then  $\partial T$ is a Euclidean cone with axis $L$.  By the angle of parallelism formula, the angle $\psi$ of the cone with the base plane $\mathbb C$ is given by $ \cosh R \sin \psi = 1$, alternatively $\sinh R = \cot \psi$.

Let $\pi$ denote perpendicular projection from $\HHH^3$ to $L$, and let 
$p_i = \pi(P_i), i=1,2$, see Figure~\ref{fig:cone}.  Let $h + i \phi$ be the complex distance between the  perpendiculars $[p_1,P_1]$ and $[p_2,P_2]$ (so that $h = d_{\HHH}(p_1,p_2)$ and   $\phi$ is the rotation angle).

We estimate $l$ by a line integral on $\partial T$.
 Take polar coordinates on $\partial T$, so that for a point $P= (t,\theta) \in \partial T$, $t>0$ denotes the Euclidean distance from $\omega$ to $P$  and $\theta$ is the angle between the plane containing $L$ and $\omega P$ and the plane containing $L$ and the real axis in $\mathbb C$.
 Since the hyperbolic metric  is conformally the same as the Euclidean metric, the element of hyperbolic arc length on $\partial T$ is
\begin{equation} 
\label{eqn:1} ds^2 = \dfrac{(t \cos \psi d\theta)^2 + dt^2}{(t\sin \psi)^2} = \sinh^2 Rd\theta^2 + \cosh^2R dt^2/t^2.
\end{equation}
 
 For $x,y>0$ we have $  \sqrt{x^2 + y^2} \asymp x+y$.
 Thus integrating~\eqref{eqn:1} gives
\begin{equation} 
\label{eqn:2}l \asymp  \phi   \sinh R+ h\cosh R \asymp e^R(h + \phi).
\end{equation}

\begin{figure}[hbt] 
 \centering
\includegraphics[height=8cm]{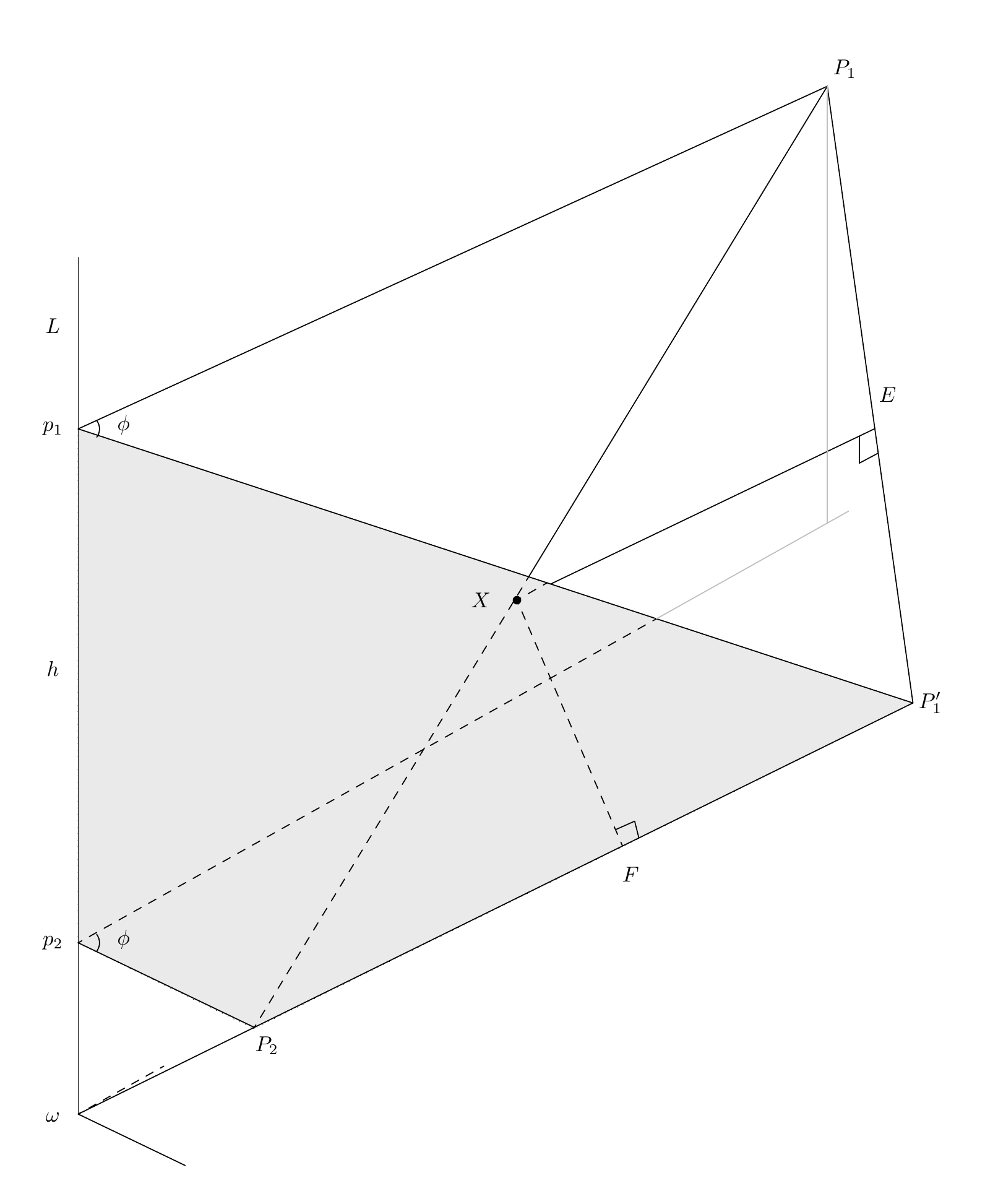} 
\caption{The estimate of $d$ in the proof of Lemma~\ref{expdist2}}
\label{fig:cone}
 \end{figure}

Now we estimate $d$.  Referring to Figure~\ref{fig:cone}, let $X$ be the midpoint of the segment $P_1P_2$, so that $d/2 = |XP_1|$, where $|AB|$ denotes the hyperbolic length of the geodesic segment from $A$ to $B$. Let $P_1'$ be the point  in the hyperbolic plane orthogonal to $L$ through $p_1$  (and thus containing $p_1$)
 such that  $\angle P_1'p_1P_1 = \phi$.  By symmetry, $  |XP_1| = |XP_1'|$, and by construction, $P_1', P_2$ and $L$ are coplanar. Let $E,F$ be the feet of the perpendiculars from $X$ to the geodesic lines
 $P_1P_1'$ and $P_1'P_2$ respectively. 
Considering the two right angled triangles $XEP_1'$ and $XFP_1'$ we find
$$ d/2 = |XP_1'| > |XE| + |EP_1'| -\mathrm{const.}>  |EP_1'| -\mathrm{const.}$$
and
$$ d/2 = |XP_1'| > |XF| + |FP_1'| -\mathrm{const.} >  |FP_1'| -\mathrm{const.}$$
so that 
$$ e^{d/2} \succ e^{ |FP_1'|} \ \ \mathrm {and} \ \ e^{d/2} \succ e^{ |EP_1'|}.$$

Computing
  in the plane through $P_1,P_1', p_1 $  we find $\sinh |EP_1'| = \sin \phi/2 \sinh R$, while computing in the plane containing $L, P_1'$ and $P_2$ we have 
$\sinh |FP_1'| = \sinh h/2 \cosh R$. 
Thus since $R$ is bounded below by $R_0$, 
$$ e^{d/2} \succ e^{ |EP_1'|} \succ \phi e^R$$ and likewise
$$ e^{d/2} \succ e^{ |FP_1'|} \succ h e^R.$$
Hence $$ e^{d/2}  \succ \max \{ \phi, h \} e^R \asymp ( \phi+ h ) e^R \asymp  l$$  which by~\eqref{eqn:2}    proves that $ e^{d/2} \prec l$.

To prove the inequality in the other direction note that provided that $h \leq h_0$ we have from the above $\sinh |EP_1'|  \prec  \phi/2 \sinh R$ and $\sinh |FP_1'|  \prec  h/2 \cosh R$.  Now $d \leq |P_1P_1'| + | P_1'P_2|$ gives $$d/2  \leq |EP_1'| + | FP_1'| \leq 2 \max (|EP_1'| , | FP_1'|).$$
Since $ d \geq d_0$ at least one of $\sinh |EP_1'| , \sinh |FP_1'| $ is bounded away from $0$ so that 
$$ e^{d/2} \leq  \max (e^{2|EP_1'| } , e^{2| Fa'| }) \asymp \max ( \phi e^R, he^R)  \asymp e^R (\phi +h).$$ The result follows from~\eqref{eqn:2}. 
\end{proof}

The next two lemmas involve the penetration of geodesics into tubes and horoballs.

\begin{appendthm}
\label{horoballs}
Suppose that  in the  ball  model $\mathbb B$, $H$ is a horoball such that $O \notin Int H$. Suppose also that  points $P_1,P_2 \in \dd H$  lie outside $B_{\HHH}(O;N)$. Then the 
 geodesic segment  $[P_1,P_2]$ lies outside $B_{\HHH}(O;  N/4 - c)$, for some universal $c>0$. 
\end{appendthm}
\begin{proof} 
First consider the case in which $O \in \partial H$.  Let $d_i = d_{\HHH}(O,P_i)$ and   let $l_i = d_{\dd H}(O,P_i)$, where as above $d_{\dd H}(.,.)$ denotes distance measured along $\partial H$.  Also let $d  = d_{\HHH}(P_1,P_2)$    let $l  = d_{\dd H}(P_1,P_2)$. By hypothesis $d_i \geq N, i=1,2$. Clearly we may fix some $d_0 >0$ and assume that $d \geq d_0$, otherwise the result is trivial. Hence by Lemma~\ref{expdist1}, $ l_i \asymp e^{d_i/2}$ and $ l  \asymp e^{d/2}$.
Thus $$e^{d/2} \asymp  l  \leq l_1+ l_2 \asymp e^{d_1/2} + e^{d_2/2} \leq 2 e^{\max(d_1,d_2)/2}$$
so that 
$d  \ladd \max (d_1,d_2)$.

\begin{figure}[hbt] 
 \centering
\includegraphics[height=5.5cm]{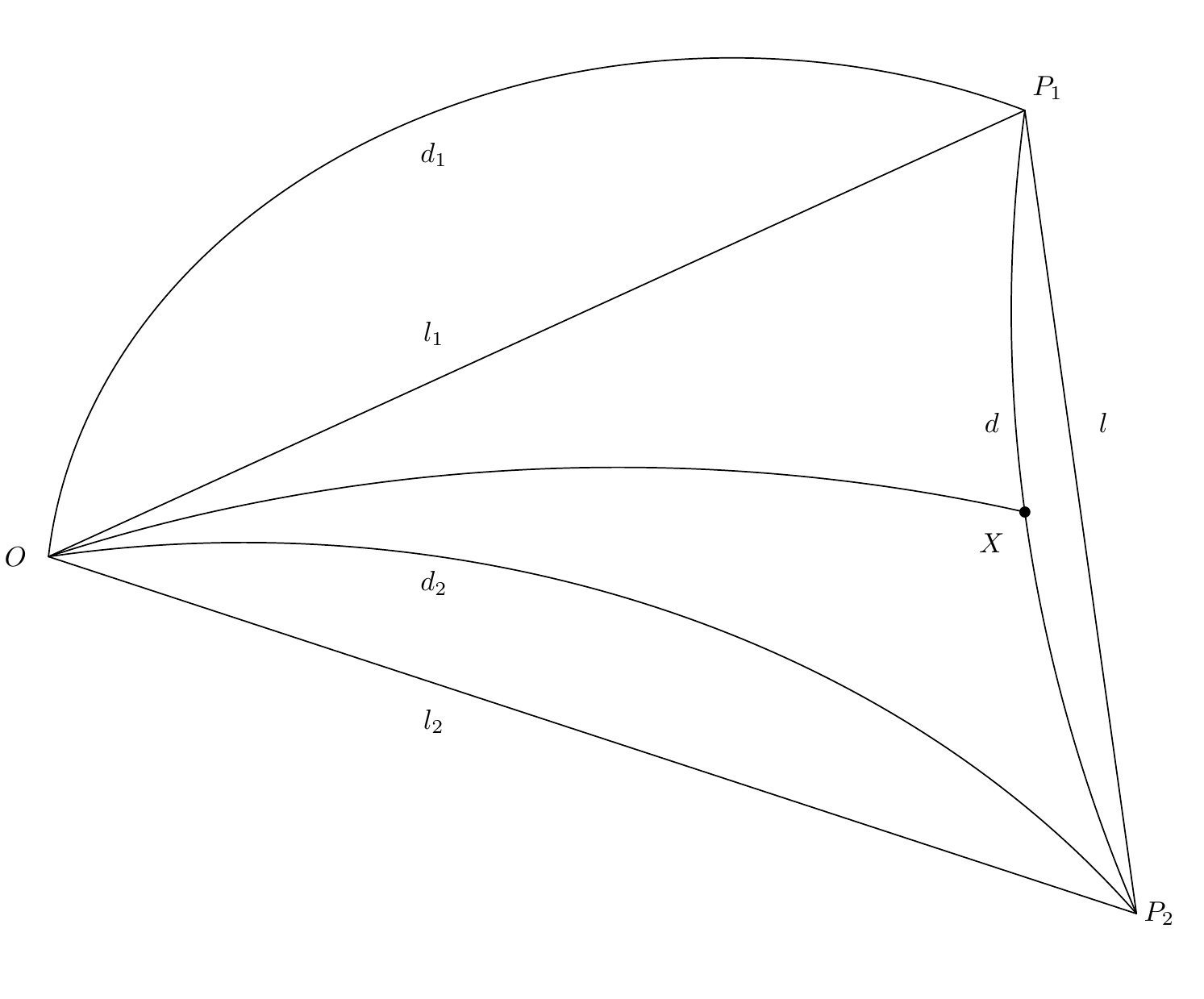} 
\caption{Configuration for Lemma~\ref{horoballs}. The points $O,P_1,P_2$ are all on the boundary of a horoball $H$.}
\label{fig:horoball}
 \end{figure}

 Assume that $ d_1 \geq d_2$.   Considering the hyperbolic triangle $OP_1P_2$  and its altitude $OX$, we have
   \begin{eqnarray*}
d = d_{\HHH}(X,P_2)+d_{\HHH}(X,P_1) & \ladd & d_1 \cr
d_{\HHH}(O,X)+d_{\HHH}(X,P_1) & \add & d_1 \cr
d_{\HHH}(O,X) +  d_{\HHH}(X,P_2) & \add & d_2 .
\end{eqnarray*}
The first two lines give  $d_{\HHH}(X,P_2) \ladd d_{\HHH}(O,X)$ and hence by the last line $ d_{\HHH}(O,X) \gadd d_2/2 \geq N/2$. Since $X$ is the closest point on $[P_1,P_2]$ to $O$, the result follows.

Finally suppose that $O \notin \dd H$. Let $O'$ be the foot of the perpendicular from $O$ to $H$. Then for any point $Y \in H$, since the angle between the geodesic segments $[O,O']$ and $[O',Y]$ is at least $\pi/2$, 
$$d_{\HHH}(O,Y) \add d_{\HHH}(O,O') + d_{\HHH}(O',Y) .$$
If $  d_{\HHH}(O,O') \geq N/2$ there is nothing to prove since by convexity the nearest point on $[P_1,P_2]$ to $O$ is in $  H$. If 
$  d_{\HHH}(O,O')< N/2$   then $  d_{\HHH}(O',P_i) \gadd  N/2$. The proof above with $N/2$ in place of $N$ then gives $  d_{\HHH}(O',[P_1,P_2]) \gadd  N/4$ so that    $  d_{\HHH}(O, [P_1,P_2]) \gadd d_{\HHH}(O,O') +  d_{\HHH}(O',[P_1,P_2]) \gadd  N/4$ as claimed.
\end{proof}

\begin{appendthm}
\label{tubes}
Suppose that $T \subset \HHH^3$ is an equidistant tube of radius $R \geq R_0$ around a geodesic $L $ in the ball model, and suppose that  $O \notin Int T$. Suppose also that  $P_1,P_2 \in \dd T$  lie outside $B_{\HHH}(O;N)$, and that in addition there is a path joining $P_1,P_2$ on $\partial T$ and  outside $B_{\HHH}(O;N)$. Then the 
geodesic segment of $[P_1,P_2]$ lies outside $B_{\HHH}(O;  N/4 -c)$  for some universal $c>0$. 
\end{appendthm}
\begin{proof}  As in the proof of Lemma~\ref{horoballs}, it will be enough to   show that  in the case   $O \in \partial T$ that   $[P_1,P_2]$   is outside $B_{\HHH}(O;N/2-c)$.

Let $\pi$ denote perpendicular projection from $\HHH^3$ onto the axis $L$ of $T$. As in the proof of Lemma~\ref{easy1}, let $p_i = \pi(P_i)$ and   write $h + i \phi$ for the complex distance between the perpendiculars $[p_i ,P_i], i=1,2$, see Figure~\ref{fig:cone}.

First suppose that the distance from $o = \pi(O)$ to the segment $  [p_1,p_2]$ is at least $1$, and suppose that $o$ is nearer to  $p_1$ than $p_2$. Let $\rho = d_{\HHH}(O, p_1)$.  Let $\Pi$ be the plane perpendicular to $L$ through $p_1$, and let $K$ be the closed half space cut off by $\Pi$ and not containing $o$. Then $K$ contains both $P_1$ and $P_2$, and hence the segment $[P_1,P_2]$; moreover
$d_{\HHH}(O,P_1) \add 2R + \rho $
while
for any point $X \in K$ we have $d_{\HHH}(O, X) \gadd R + \rho$. Since $d_{\HHH}(O,P_1) \geq N$ it follows that $R + \rho/2 \gadd N/2$ so that 
$d_{\HHH}(O,  [P_1,P_2]) \geq R + \rho \geq  N/2$.

Now suppose that $o$ is at distance at most $1$ to the segment $  [p_1,p_2]$, and suppose also that   $|h| \geq 1$. The hyperbolic geodesic $\alpha$ from $P_1$ to $P_2$ is at distance at most $c$ to the union of the geodesic segments $[P_1,p_1], [p_1, p_2], [p_2, P_2]$. Let $\Pi'$ be the plane orthogonal to $L$ containing $O$, so that $o \in \Pi'$.  Then $\Pi'$ separates $T$ and the points $P_1,P_2$ are in opposite sides of $\Pi'$. Hence the projection of any path from  $P_1$ to $P_2$
on $\dd T$ onto $L$ must contain $o$. Let $o' \in \pi^{-1}(o)  \in \beta$, where $\beta$ is a path from  $P_1$ to $P_2$ on $\dd T$.   
Note $o$ is the centre of a circle of radius $R$ whose boundary contains both $O$ and $o'$. Hence $R \geq d_{\HHH}(O,o')/2$. By hypothesis since $o' \in \beta$ we have 
$d_{\HHH}(O,o') \geq N$, so $R \geq N/2$.
On the other hand, since  $\alpha$ tracks $[P_1,p_1] \cup [p_1, p_2] \cup  [p_2, P_2]$ within distance $c$ for some universal $c>0$, since $o \in [p_1, p_2]$, and since $o$ is the nearest point to $O$ on $L$, we have that $d_{\HHH}(O,[P_1,P_2]) \add d_{\HHH}(O,o)  = R$. Thus $d_{\HHH}(O,[P_1,P_2]) \gadd N/2$ as claimed.

Finally suppose that  $|h| \leq 1$.  Let $d = d_{\HHH}(P_1,P_2)$,  $l = d_{\dd T}(P_1,P_2)$, $d_i = d_{\HHH}(O,P_i) \geq N$  and $l_i = d_{\dd T}(O,P_i), i=1,2$.  
Picking $h_0 = 2$ in Lemma~\ref{tubes} we have $ l \asymp e^{d/2}$, and in addition, since $o$ is at distance at most $1$ to the segment $  [p_1,p_2] \subset L$, we   have $d_{\HHH}(o,p_i) =  d_{\HHH}(\pi(O),\pi(P_i) \leq h_0$ so that $ l_i \asymp e^{d_i/2}$. 
 Then exactly the same proof as in Lemma~\ref{horoballs} gives that $  d_{\HHH}(O,[P_1,P_2]) \gadd  N/2$.
This completes the proof.
\end{proof}
 
\begin{appendremark} {\rm Lemma~\ref{tubes} required a hypothesis not needed in Lemma~\ref{horoballs}, namely that  the shortest path on $\partial T$ from  $P_1$ to $P_2$  is outside $B_{\HHH}(O;N)$.  This is only used in the case
 in which $\pi(O)$ is near to $[\pi(P_1), \pi(P_2)]$ and $|h| \geq 1$, however here it is crucial. To see this consider the situation in which $P_1, P_2 $  and $L$ are coplanar  and  $O$ is the midpoint of the path from $P_1$ to $P_2$  on $\dd T$ (contrary to the hypothesis under discussion). Then the geodesic  $[P_1, P_2]$ tracks $[P_1,p_1]\cup  [p_1, p_2]\cup [p_2, P_2]$ so that  $d_{\HHH}(O,[P_1, P_2] ) \add R$ but at the same time we could have   $N \gg R$. 
This would cause problems in Section~\ref{sec:algebraic} when we need to find uniform estimates for a sequence of groups with fixed $R_0$ but $N \to \infty$.} 
\end{appendremark}

\bibliography{cmotions}

\begin{thebibliography}{000}

 

\bibitem{BH} M. Bridson and A. Haefliger. \newblock Metric spaces of non-positive curvature.
\newblock {\em Springer Grundlehren} 319, Springer, Berlin, 1999.

 \bibitem{brock-itn} J. Brock. \newblock Iteration of mapping classes and limits of hyperbolic 3-manifolds.
\newblock {\em Inventiones Math.} 1043, 523 -- 570, 2001.


 \bibitem{canary} R. Canary.
       \newblock Ends of hyperbolic $3$-manifolds.
       \newblock {\em J. Amer. Math. Soc.}  6, 1 -- 35, 1993.


 \bibitem{CT} J. Cannon and W. P. Thurston.
       \newblock Group {I}nvariant {P}eano {C}urves.
       \newblock {\em Geometry and Topology}  11, 1315 -- 1355, 2007.
       

\bibitem{evans1} R. Evans.
 \newblock Deformation spaces of hyperbolic 3-manifolds: strong convergence and
tameness.
\newblock {\em Ph.D. Thesis, University of Michigan}, 2000.

 
 \bibitem{evans2} R. Evans. \newblock Weakly type-preserving sequences and strong convergence.
\newblock {\em Geometriae Dedicata} 108, 71 -- 92, 2004.

\bibitem{fenchel-n} W. Fenchel and J. Nielsen. \newblock Discontinuous groups of isometries in the hyperbolic plane.
\newblock {\em de Gruyter Studies in Math.} 319, Berlin 2003.


\bibitem{floyd} W. Floyd.
\newblock Group completions and limit sets of Kleinian groups.
\newblock {\em Inventiones Math.} 57, 205 -- 218, 1980.


\bibitem{franca} S. Francaviglia.
\newblock Constructing equivariant maps for representations.
\newblock {\em Ann. Inst. Fourier} 59, 393 -- 428, 2009.

\bibitem{GhH} E. Ghys and  P. de la Harpe (eds.).
\newblock  Sur les groupes hyperboliques d'apr\`{e}s {M}ikhael {G}romov.
       \newblock {\em Progress in Math. Vol 83, Birkhauser, Boston}, 1990.


\bibitem{kerckhoff-thurston} S. Kerckhoff and W. Thurston.
\newblock  {Non-continuity of the action of the modular group at the
         Bers' boundary of Teichm\"uller Space.}
\newblock {\em Inventiones Math.} 100, 25 -- 48, 1990.

\bibitem{kuusalo} T. Kuusalo.
\newblock  {Boundary mappings of geometric isomorphisms of Fuchsian groups.}
\newblock {\em Ann. Acad. Sci. Fennicae Ser. A Math.}  545, 1 -- 7, 1973.



\bibitem{jor-mar} T. J{\o}rgensen and A. Marden. 
\newblock Algebraic and geometric convergence of
                {K}leinian groups.
                \newblock {\em Math. Scand.} 66, 47 -- 72, 1990.

\bibitem{MSS} R.~Ma{\~n}\'e, P.~Sad  and D.~Sullivan. 
\newblock {On the dynamics of rational maps.}
\newblock {\em Ann. Sci. \'Ecole. Norm. Sup.}, 16, 193 -- 217, 1983.

\bibitem{marden-book} A.~Marden.
\newblock {\em Outer Circles: An introduction to hyperbolic $3$-manifolds.}
\newblock Cambridge University Press, 2007.

\bibitem{ctm-locconn} C. T. McMullen.
 \newblock Local connectivity, {K}leinian groups and geodesics on the blow-up of the torus.
\newblock {\em  Inventiones Math.} 97, 95 -- 127, 2001.


\bibitem{MinskyOPT}
Y. N. Minsky.
\newblock The classification of punctured-torus groups.
\newblock {\em Ann. of Math.} 149,  559 -- 626, 1999.


\bibitem{minsky-elc1} Y. N. Minsky.
\newblock The {c}lassification of {K}leinian surface groups {I}: Models and Bounds.
\newblock {\em Ann. of Math.} 171,  1 -- 107, 2010.

\bibitem{mitra-ct} M. Mitra.
\newblock Cannon-{T}hurston maps for hyperbolic group extensions.
\newblock {\em Topology} 37, 527 -- 538, 1998.

 \bibitem{miyachi} H. Miyachi.
\newblock Moduli of continuity of Cannon-Thurston maps.
\newblock {\em In Spaces of Kleinian groups; London Math. Soc. Lecture Notes 329, Cambridge University Press}, 121 --150, 
2006.   


\bibitem{mahan-kl} M. Mj.
\newblock Cannon-Thurston Maps for Kleinian Groups.
\newblock {\em preprint, arXiv:1002.0996}, 2010.




          
  \bibitem{mjs2} M. Mj and C. Series.
\newblock Limits  of limit sets II.
\newblock {\em In preparation.}
          
    


\bibitem{nielsen} J. Nielsen.  
 \newblock 	Untersuchungen zur Topologie der geschlossenen zweiseitigen Fl\"achen. 
\newblock {\em Acta Math.} 50, 189 -- 358, 1927.


 \bibitem{tukia} P. Tukia.
\newblock On isomorphisms of geometrically finite M\"obius groups. 
\newblock {\em IHES Publ.} 61, 127 -- 140, 1985.

 \bibitem{tukia1} P. Tukia.
\newblock A remark on a paper by Floyd.
\newblock {\em In Holomorphic functions and moduli Vol II; MSRI Publ. 11, Springer, New York}, 165 --172, 
1988.      


\end{thebibliography}
\bibliographystyle{alpha}

\end{document}